\documentclass[12pt]{amsart}
\usepackage{amssymb, amstext, amscd, amsmath, color, xypic}

\usepackage[dvips]{graphicx}

\makeatletter
\def\@cite#1#2{{\m@th\upshape\bfseries%
[{#1\if@tempswa{\m@th\upshape\mdseries, #2}\fi}]}} \makeatother
\theoremstyle{plain}
\newtheorem{thm}[subsection]{Theorem}

\newtheorem{prop}[subsection]{Proposition}
\newtheorem{lem}[subsection]{Lemma}
\theoremstyle{definition}
\newtheorem{rem}[subsection]{Remark}
\newtheorem{example}[subsection]{Example}

\newtheorem{defn}[subsection]{Definition}

\newcommand{\bN}{{\mathbb{N}}}

\newcommand{\bR}{{\mathbb{R}}}

\newcommand{\bT}{{\mathbb{T}}}

\newcommand{\bZ}{{\mathbb{Z}}}

  \newcommand{\F}{{\mathcal{F}}}
  \newcommand{\G}{{\mathcal{G}}}

  \newcommand{\K}{{\mathcal{K}}}

  \newcommand{\U}{{\mathcal{U}}}
  \newcommand{\V}{{\mathcal{V}}}

\newcommand{\fV}{{\mathfrak{V}}}

\newcommand{\upchi}{{\raise.35ex\hbox{\ensuremath{\chi}}}}

\newcommand{\lip}{\langle}
\newcommand{\rip}{\rangle}

\newcommand{\ul}{\underline}

\newcommand{\Ltwo}{L^2(\bR)}
\newcommand{\Ltwotwo}{L^2(\bR^2)}
\newcommand{\Ltwod}{L^2(\bR^d)}
\newcommand{\Ltwohatd}{L^2(\hat{\bR}^d)}

\newcommand{\hrmra}{(\phi, \fV)}
\newcommand{\bmra}{(\phi, \{V_{(i,j)}\})}
\newcommand{\gld}{\mathrm{GL}_d(\bR)}

\begin{document}

\title[Higher Rank Wavelets]{Higher Rank  Wavelets}

\author[S. Olphert]{Sean Olphert}\thanks{The first author is
supported by an EPSRC grant}
\address{Dept.\ Math.\ Stats.\\ Lancaster University\\
Lancaster LA1 4YF \\U.K. } \email{s.olphert@lancaster.ac.uk}

\author[S.C. Power]{Stephen C. Power}
\thanks{The second author is partially supported by
EPSRC grant EP/E002625/1 }
\address{Dept.\ Math.\ Stats.\\ Lancaster University\\
Lancaster LA1 4YF \\U.K. } \email{s.power@lancaster.ac.uk}

\begin{abstract}
A theory of higher rank multiresolution analysis is given in the
setting of abelian multiscalings. This theory enables the
construction, from a higher rank MRA, of  finite wavelet sets
whose multidilations have translates forming an orthonormal basis
in $L^2(\bR^d)$. While tensor products of uniscaled MRAs provide
simple examples we construct many nonseparable higher rank
wavelets. In particular we construct \textit{Latin square
wavelets} as rank $2$ variants of Haar wavelets. Also we construct
nonseparable scaling functions for rank $2$ variants of Meyer
wavelet scaling functions, and we construct the associated
nonseparable wavelets with compactly supported Fourier transforms.
On the other hand we show that compactly supported scaling
functions for biscaled MRAs are necessarily separable.
\end{abstract}

\date{}
\maketitle

\section{Introduction}\label{S:intro}

The term multiscaling in wavelet theory commonly refers to the
various scaling levels present in a nest of subspaces
$$\dots
\subseteq V_{-1} \subseteq V_0 \subseteq V_1 \subseteq \dots
$$ of
a multiresolution analysis in $L^2(\bR^d)$. See, for example,
Dutkay and Jorgensen \cite{DutPal}. However such subspaces are
associated with the powers of a {single} dilation matrix. In
contrast we develop here a theory of wavelets for higher rank
multiresolution analyses which are generated by several
independent commuting dilation matrices.

Recall that a wavelet set, or multiwavelet, is generally taken to
be a set of
 functions
 $\psi_1(x),\dots ,\psi_t(x)$ in $L^2(\bR^d)$ for which certain
 translates of dilates form an orthonormal basis of the form
\[
\{(\det A)^{-\frac{m}{2}}\psi_i(A^mx + k): m\in \bZ, k \in \bZ^d,
1\le i \le t\},
\]
where $A$ is a scaling matrix in $GL(\bR^n)$. Wavelet theory is
concerned with identifying constructions for which the wavelets
$\psi_i$ exhibit forms of directionality and smoothness. In
particular it has been of interest to obtain multivariable
wavelets which are in some sense nonseparable with respect to the
coordinates of $\bR^d$. See for example  \cite{BelWan},
\cite{CohDau}, \cite{GroMad}, \cite{HeLai}, \cite{KovVet},
\cite{Li}, \cite{Wan}. In particular Belogay and Wang
\cite{BelWan} construct nonseparable wavelets in $\bR^2$, for some
dilation matrices with determinant $2$, which are arbitrarily
smooth.

Most commonly, particularly in multiresolution analysis, the
dilation group is singly generated, as above. Some recent studies
with multiscalings that go beyond this are summarised in Gu,
Labate, Lim, Weiss and Wilson \cite{GuoLabLimWei}. In these
settings volume-preserving sheering unitaries combine with a
single strict dilation to generate the nonabelian dilation groups
of interest, and associated wavelet sets are constructed. Also
these so-called affine orthonormal systems exhibit specific forms
of directionality and nonseparability. In contrast to this our
wavelets are constructed for a dilation representation of the
abelian group $\bZ^r$.

There are two main approaches to constructing wavelet sets, namely
the multiresolution analysis approach of Mallat \cite{Mal}, on the
one hand, and constructions associated with self-similar tilings
of $\bR^d$ on the other. See Wang \cite{Wan}, for example, for
connections with tilings.
 As far as the authors are aware there
has been little development of traditional multiresolution
analysis wavelet theory for multiscaled settings in which the role
of dilation matrices $A^m$ is played by an abelian group of
dilation matrices $A_1^{m_1}A_2^{m_2}\dots A^{m_s}_r$.
 The simplest such
multiscaled context of this kind is the biscaled dyadic case for
wavelets in $\Ltwotwo$ associated with the dilation matrices
$$A=\begin{bmatrix}
2&0\\0&1\end{bmatrix}, B=\begin{bmatrix} 1&0\\0&2\end{bmatrix}.
$$
If $\psi_A(x)$ and $\psi_B(x)$ are univariate dyadic wavelets in
$\Ltwo$ then the separable function $\psi(x,y) =
\psi_A(x)\psi_B(y)$ is a biscaled wavelet, that is the set
\[
\{2^{-(m+n)/2}\psi(A^mB^nx + k): (m,n)\in \bZ^2, k \in \bZ^2\},
\]
is an orthonormal basis in $L^2(\bR^2)$.

The evident nonlocality of such separable wavelets has possibly
not encouraged the elaboration of a multiscaled wavelet theory.
(See, for example, the discussion in \cite{Woj}.)  However, we
shall show that even in this apparently adverse setting of
separated coordinates it is possible to construct nonseparable
wavelets and even nonseparable scaling functions. Moreover we
develop what might be termed a theory of higher rank wavelets and
multiresolution analysis.

We define  a higher rank multiresolution analysis $(\phi , \fV)$,
of rank $r$, where $\fV = \{V_{\ul{i}}:\ul{i} \in \bZ^r\}$ is a
commuting lattice of closed subspaces of $\Ltwod$ with appropriate
inclusions. See Definition \ref{def:bmra} and Section 2.8.
 In our first main result Theorem \ref{W:waveletthm}
 we show how one may construct  multiscaled wavelet sets from a higher rank
multiresolution of rank two.

The distinguished novelties of higher rank multiresolution
analysis  are already present  in the simplest setting of
biscaling ($r = 2$) and dimension 2 ($d =2$).  For such a biscaled
multiresolution analysis (BMRA)  the scaling function $\phi(x,y)$
possesses two marginal filter functions $m_\phi^A(\xi),
m_\phi^B(\xi)$ and these must satisfy the intertwining relation
\begin{equation}
m_\phi^A(B\xi)m_\phi^B(\xi) = m_\phi^A(\xi)m_\phi^B(A\xi).
\end{equation}
This filter relation follows from the coincidence of the tripe
subspace inclusions associated with  the dilation pairs $B, AB$
and $A, BA$. That is, the intertwining relation is a consequence
of the lattice structure of $\fV$. In addition to this the
orthogonality structure, or \textit{commuting} (projection
lattice) structure, of the BMRA leads to a filter identity of some
complexity namely,
\medskip

\[
m^A_\phi(\xi_1+\pi,2\xi_2) m^A_\phi(\xi_1,\xi_2)\overline{
m^B_\phi(\xi_1,\xi_2)m^B_\phi(2\xi_1,\xi_2+\pi)} -
\]
\[
m^A_\phi(\xi_1,2\xi_2)m^A_\phi(\xi_1+\pi,\xi_2)\overline{
m^B_\phi(\xi_1+\pi,\xi_2)m^B_\phi(2\xi_1,\xi_2+\pi)}-
\]
\[
 m^A_\phi(\xi_1+\pi,2\xi_2) m^A_\phi(\xi_1,\xi_2+\pi)\overline{
m^B_\phi(\xi_1,\xi_2+\pi)m^B_\phi(2\xi_1,\xi_2)} +
\]
\[
m^A_\phi(\xi_1,2\xi_2)m^A_\phi(\xi_1+\pi,\xi_2+\pi)\overline{
m^B_\phi(\xi_1+\pi,\xi_2+\pi)m^B_\phi(2\xi_1,\xi_2)} =0.
\]
\medskip

These two necessary conditions present challenges for the
construction of nonseparable scaling functions. Nevertheless, in
another of our main results, Theorem \ref{t:rank2meyer}, we
construct examples of BMRAs for which the scaling function is
indeed nonseparable. In this case the derived wavelets are
similarly nonseparable. Furthermore, they inherit smoothness from
$\phi$. In this way we obtain singleton bidyadic wavelets whose
Fourier transforms have compact support and these wavelets are in
fact  higher ranks variants of the well-known Meyer wavelets
\cite{Mey}.

On the other hand, even for separable rank 2 multiresolutions and
separable scaling functions we show that there is sufficient
freedom in the nondyadic case
$$(\det A-1)(\det B-1) \ge 2$$ to derive nonseparable wavelets.
This is achieved by constructing filter matrix functions which are
not elementary tensor products. The construction of these filter
functions parallels the well-known method of unitary matrix
completion although our arguments, given in the proof of Theorem
\ref{W:waveletthm}, involve a nesting of several Gram Schmidt
completion processes. The wavelet sets here include some very
interesting and computable examples of what we term \textit{Latin
square wavelet sets}, for evident reasons. (See Theorem
\ref{t:latinsquare}.) These are multiscaled versions of the
classical wavelets associated with Haar bases.


In another of our main results we  reveal a striking constraint
for compactly supported scaling functions in the biscaled theory
in $L^2(\bR^2)$, namely that  such functions are necessarily
separable. Equivalently put, a BMRA $(\phi , \fV)$ with compactly
supported scaling function $\phi$ is equivalent to an elementary
tensor of two uniscaled MRAs. This fact may have been a further
implicit obstacle to the development of multiscaled
multiresolution wavelets.

There are many intriguing wavelet directions that now seem to
beckon. For example, we have not particularly addressed smoothness
and approximation properties in this article.  Although we have
shown that the generalised Meyer context, with scaling functions
with compactly supported Fourier transforms, allows the appearance
of nonseparable wavelets, the context is nevertheless a
constraining one. Also, we have shown that the intertwining
relation is an exactitude that rules out compactly supported
higher rank scaling functions. It is plausible that the setting of
frames and higher rank GMRAs, which relaxes this condition in some
way, may allow for the construction of compactly supported higher
rank frames. Moreover it will also be of interest to develop this
theory further for dilation representations of $\bZ^d$ on $\bR^n$
other than the basic ones we consider.

 Even in theoretical articles such as this one it is
customary to make a few remarks concerning the potential
efficiency of any new species of wavelets or frames. In the
current  climate of diverse and burgeoning applications this
hardly seems necessary. However, we remark that nature often
presents pairs of partially independent features with their own
scaling aspects. Space-time scales are one obvious source of this.
A purely spatial context can be found in the statistical theory of
textures. This theory is suitable for localised analysis through
wavelets (see \cite{EckNasTre}) and, evidently, many natural
textures have a rectangular emphasis.


Our account is self-contained with complete proofs apart from a
few basic standard lemmas. In particular in Section we reprove the
construction and formula of a uniscaled dyadic wavelet associated
with an MRA, and in Section 7 we give the construction scheme for
classical Meyer wavelets. We refer the reader to the books of
Wojtaszczyk \cite{Woj} and Brattelli and Jorgensen \cite{BraJor}
which are excellent sources for diverse wavelet theory.

\section{Higher Rank Multi-resolution Analysis}

In this section we formulate the structure of a higher rank
multiresolution analysis although we focus attention on the
biscaled case of Definition \ref{def:bmra}. We obtain the
frequency domain identification of a function $f(x)$ in $V_{i,j}$
in terms of a filter function $m_f(\xi)$ and the Fourier transform
$\hat{\phi}(\xi)$. As preparation for the construction of wavelets
from a BMRA we obtain a filter function  matrix criterion in order
that a given set $\{f_1,\dots , f_r\}$ of unit vectors in $\Ltwod$
should generate an orthonormal basis under the operations of
translation and bidilation.

\subsection{Preliminaries}
We take the Fourier transform $\hat{f}$ of an integrable function
$f(x)$ on $\bR^d$ in the form
\[
\hat{f}(\xi)=\frac{1}{(2\pi)^{\frac{d}{2}}}\int_{\bR^d}e^{-i\lip
\xi,x \rip }f(x)dx,
\]
and define $\hat{f}$ for $f$ in $\Ltwod$ by the usual unitary
extension of the map $f \to \hat{f}$. The domain of Fourier
transforms is referred to as the frequency domain and for emphasis
is denoted $\hat{\bR}^d$. For a matrix $A \in GL_d(\bR)$ and a
point $x$ in the time domain $\bR^d$ write $Ax$ the product with
$x$ viewed as a column vector and define the unitary operator $D_A
: \Ltwod \to \Ltwod$ by
\begin{equation}
(D_Af)(x) = |\det A|^{1/2}f(Ax).
\end{equation}

Recall that  $A$ is said to be an \emph{expansive} matrix if all
eigenvalues of $A$ have absolute value greater than one. We say
here that $A$ is a \emph{dilation matrix} if in addition
$A(\bZ^d)\subset\bZ^d$, that is, if $A$ is a matrix of integers.

A \emph{commuting dilation pair} $(A,B)$ is a pair of matrices in
$ GL_d(\bR)$ such that, after a suitable permutation of the
standard basis of $\bR^d$ we have
\begin{align}
A&=\begin{bmatrix} A'&0\\0&I_{d-p}\end{bmatrix},&
B=\begin{bmatrix} I_{d-q}&0\\0&B'\end{bmatrix}\label{CDP1}
\end{align}
 where
$A'$ and $B'$ are dilation matrices in $GL_p(\bR)$ and
$GL_q(\bR)$. Note that $T= A^nB^m$ is a dilation matrix if and
only if $n\ge 1$ and $m \ge 1$, while the matrices $A^i$ and $B^j$
are {\it weak dilation matrices} in the sense that their
eigenvalues are not less than $1$. In Sections 5, 6 and 7 we
confine attention to wavelets in $\Ltwod$ for a fundamental
dilation pair
\[
 A
 =\begin{bmatrix}
\alpha I &0\\0& I\end{bmatrix}, \quad \quad B= \begin{bmatrix}
I&0\\0&\beta I\end{bmatrix}
\]
where $\alpha , \beta \ge 2$ are integers.

\subsection{Biscaled Multiresolution Analysis}

We now define a biscaled  analogue of a multiresolution analysis
associated with a single dilation matrix.

\begin{defn}\label{def:bmra}
A biscaled multiresolution analysis, or BMRA, with respect to the
commuting dilation pair $(A, B)$ in $\gld$  is pair $\hrmra$ where
$\phi \in \Ltwod$ and $\fV = \{V_{i,j} : (i,j) \in \bZ^2\}$ is a
family of closed subspaces of $\Ltwod$ such that
\medskip

(i) $V_{i,j} \subseteq V_{l,m}$,  for  $ i \leq l,$ ~$ j \leq m$,

(ii) $D_A^iD_B^jV_{0,0} = V_{i,j}$, for all $i,j$,

(iii) $\bigcap_{(i,j)\in\bZ^2} V_{i,j}= \{0\},$

(iv) $\overline{\bigcup_{(i,j)\in\bZ^2} V_{i,j}} =\Ltwod$,

(v) $P_{i,j}P_{l,m}=P_{l,m}P_{i,j}$ for all orthogonal projections
$P_{s,t}:\Ltwod \to V_{s,t}$.

(vi)  $\{\phi(x-k):k\in\bZ^d\}$
 is an orthonormal basis for  $V_{0,0}$.
\end{defn}
\medskip

It is helpful to make the following distinctions.

If we have merely conditions (i), (iii), (iv) then we say that
$\fV$ is a \textit{grid} of subspaces. If, additionally,
$V_{i,j}\cap V_{k,l}=V_{m,n}$ with $m=min\{i,k\}$, $n=min\{j,l\}$,
for all $i,j,k,l$, then we say that $\fV$ is a \textit{lattice} of
subspaces. Finally, if, additionally, the stronger commuting
condition (v) holds,then we say that $\fV$ is a \textit{commuting
lattice} of subspaces. Thus a BMRA is a commuting lattice of type
$\bZ \times \bZ$ which is generated by a single function $\phi$
and the dilation operators.

We begin the analysis of BMRAs by identifying the various
frequency domain descriptions of functions in $V_{i,j}$.

Let $f \in V_{i,j}$, $T= A^iB^j$ and let $t = \det T$. Then
$f\circ T^{-1}$ is in $V_{0,0}$ and so there is an expansion
\[
f(T^{-1}x)= \sum_{k\in \bZ^d}c_f(k) \phi(x-k)
\]
in $\Ltwod$ for some square summable   sequence $\{c_f(k)\}_{k\in
\bZ^d}$. This sequence provides the \textit{filter coefficients}
of $f$ for the dilation $T$.  Define the $2\pi \bZ^d$-periodic
function $m_f(\xi)$ in $\Ltwohatd$ by
\[
m_f(\xi) = \frac{1}{t}\sum_{k\in \bZ^d}c_f(k)e^{-i\lip \xi,k\rip}.
\]
This function is the \textit{filter function} for $f$ associated
with $T$. Note that
 $\widehat{f\circ T^{-1}}(\xi) = \hat{f}(T^*\xi)$ and on the other hand
\[
 \widehat{(f\circ T^{-1})}(\xi)=\sum_{k\in \bZ^d}c_f(k)e^{-i
\lip \xi, k \rip}\hat{\phi}(\xi).
\]
Thus $\hat{f}(T^*\xi)= m_f(\xi)\hat{\phi}(\xi)$ and so the filter
function determines $f$ in $TV_{0,0}$ by means of the formula
\begin{equation}\label{e:filter}
\hat{f}(\xi) = m_f(T^{*-1}\xi)\hat{\phi}(T^{*-1}\xi).
\end{equation}
We may view the filter function $m_f(\xi)$ as being defined on
$2\pi\bT^d$, in which case one readily sees that
\[
\|m_f\|_{L^2(2\pi\bT^d)} = \frac{1}{\sqrt{t}}\|f\|_2.
\]
In view of our applications and for notational simplicity we
assume henceforth that $A=A^*, B= B^*$ and hence $T=T^*$.

Note that $m_f(T^{-1}\xi)$ is periodic by translates from $2\pi
T\bZ^d$ and that we have the following converse. If $g(\xi)$ is
$2\pi T\bZ^d$ periodic and has square integrable restriction to
$2\pi \bT^d$ then the function $f$ in $\Ltwod$ with $\hat{f}(\xi)
= g(\xi)\hat{\phi}(T^{-1}\xi)$ lies in $V_{i,j}.$ In particular we
have the following criterion for membership of $V_{0,0}$ which we
state explicitly as it will play a role in the proof of the
lattice property of a BMRA.

\begin{prop}\label{p:V00}
Let $(\phi, \fV)$ be a BMRA. Then a function $f$ in $L^2(\bR^d)$
lies in $V_{0,0}$ if and only if $\hat{f}(\xi) =
g(\xi)\hat{\phi}(\xi)$ for some $2\pi \bZ^d$ periodic function $g$
which has square summable restriction to $2 \pi \bT^d$.
\end{prop}

Consider now the cosets of $T\bZ^d$ in the abelian group $\bZ^d$,
of which there are $t$ in number, say $E_0,\dots ,E_{t-1}$. (See
Proposition 5.5 in \cite{Woj} for example.) Representative
elements $\Gamma_0, \dots ,\Gamma_{t-1}$ of the cosets are often
known as \textit{digits} and comprise a \textit{set of digits for}
$T$.

For $0\leq i \leq t-1$ define the $i^{th}$ translate function for
the filter $m_f(\xi)$ as the function
\begin{equation}
m_f^i(\xi)=m_f(\xi+2\pi T^{-1} \Gamma_i).
\end{equation}
Note that from the norm correspondence above it follows that if
$f$ is a unit vector then the row vector formed by the $t$
translates $m_f^i(\xi)$ is a unit vector almost everywhere on
$2\pi\bT^d$.

\begin{defn}\label{d:cosetmatrix}
Let $(\phi, \fV)$ be a biscaled multiresolution analysis, let
($i,j) \in \bZ^2$, and let $\F =\{f_0,\dots ,f_s\}$ be an ordered
set of functions in $V_{i,j}$. The (translation form) filter
matrix for $\F$ is the function matrix $U_\F(\xi)$ where
\[
U_\F(\xi) = [m^i_{f_l}(\xi)]_{l=0,i=0}^{s, t-1}
\]
where $t = \det A^iB^j$.
\end{defn}

We remark that in Section 4 we define a related filter matrix
$U'_\F(\xi)$ whose rows are formed by the partial sums over cosets
of the filter functions in the first column of $U_\K$.

The key filter matrix lemma below  will be used in the
construction of biscaled wavelets and in this connection  it will
be applied to various pairs $X, Y$ from the spaces $V_{0,0},
V_{1,0}, V_{0,1}, V_{1,1}$. The proof is similar to well known
constructions in the the monoscaled case.

Let $X \subseteq Y$ be closed subspaces of $\Ltwod$ with $Y=D_TX$
and let $\phi$ be a scaling function for $X, Y$ in the sense that
the translates of $\phi$ form an orthonormal basis for $X$.
 In the light of our
applications later we make the simplifying assumption that the
digits are chosen so that the set $T[0,2\pi]^d$ is the essentially
disjoint union of the sets $[0,2\pi]^d+ 2\pi \Gamma_i, i = 0,\dots
,t-1$.

First we note the following well known simple property of scaling
functions which is equivalent to the orthonormality of translates.
This will feature in the proof of Lemma \ref{l:filtermatrix} and
in the construction of higher rank Meyer wavelets.

\begin{lem}\label{l:onset}
The set of $\bZ^d$-translates of a function $\phi$ in $\Ltwod$
forms an orthonormal set if and only if
\[
\sum_{k\in\bZ^d}|\hat{\phi}(\eta-2\pi k)|^2=\frac{1}{(2\pi)^{d}}.
\]
\end{lem}

\begin{lem}\label{l:filtermatrix}
Let $(\phi, X, T, Y, \{\Gamma_i\})$ be as above and let $\F =
\{f_0,\dots ,f_s\}$ be a finite ordered set in $Y$ with filter
matrix $U_\F(\xi)$. Then the system
$\{f_l(x-k)\}_{k\in\bZ^d,l=0,...,s}$ is an orthonormal set in $Y$
if and only if for almost every $\xi$ the matrix $U_\F(\xi)$
 is a partial isometry with full range. Furthermore the system is
 an orthonormal basis for $Y$
if and only if $s=t-1$ and for almost every $\xi$ the matrix
$U_\F(\xi)$ is unitary.
\end{lem}

\begin{proof}
Let $f, g$ be two functions in $\F$. We evaluate the inner product
$I = \lip f(x-k_1),g(x-k_2)\rip$ for $k_1, k_2\in\bZ^d$. Since
$f(x-k)$ has Fourier transform $e^{-i\lip\xi,k\rip}\hat{f}(\xi)$
the unitarity of the Fourier transform implies
\begin{gather}
I=\int_{\hat{\bR}^2} \hat{f}(\xi)\overline{\hat{g}(\xi)}e^{-i\lip
\xi,k_1-k_2\rip }d\xi. \intertext{Thus using \eqref{e:filter}, the
substitution $\eta=T^{-1}\xi$, and the $2\pi\bZ^d$ periodicity of
the filter functions, we have}
 I=
\int_{\hat{\bR}^d}m_f(T^{-1}\xi)\overline{m_g(T^{-1}\xi)}
|\hat{\phi}(T^{-1}\xi)|^2e^{-i\lip\xi,k_1-k_2\rip
} d\xi
\\
=t\int_{\hat{\bR}^d}m_f(\eta)\overline{m_g(\eta)}|\hat{\phi}(\eta)|^2e^{-i\lip
T\eta,k_1-k_2\rip} d\eta
\\
= \sum_{k\in\bZ^d}
t\int_{[0,2\pi]^d}m_f(\eta)\overline{m_g(\eta)}|\hat{\phi}(\eta
-2\pi k)|^2e^{-i\lip T\eta,k_1-k_2\rip} d\eta.
\\
\intertext{By the last lemma,}
\sum_{k\in\bZ^d}|\hat{\phi}(\eta-2\pi k)|^2=\frac{1}{(2\pi)^{d}}
\intertext{and so}
I=\frac{t}{(2\pi)^{d}}\int_{[0,2\pi]^d}m_f(\eta)
\overline{m_g(\eta)}e^{-i\lip T\eta,k_1-k_2\rip }d \eta
\intertext{thus}
I=\frac{1}{(2\pi)^{d}}\int_{T[0,2\pi]^d}m_f(T^{-1}\xi)
\overline{m_g(T^{-1}\xi)}e^{-i\lip \xi,k_1-k_2\rip }d \xi.
\intertext{Considering translates by $2\pi\Gamma_i$, and using
periodicity in the exponential factor we  conclude that the inner
product $\lip f(x-k_1),g(x-k_2)\rip$ is equal to}
\frac{1}{(2\pi)^{d}}\int_{[0,2\pi]^d} \left(\sum_{i=0}^{t-1} m^i_f
(T^{-1}\xi)\overline{m^i_g(T^{-1}\xi)}\right)e^{-i(\lip \xi
,k_1-k_2\rip)}d\xi.
\end{gather}
Now the sum function in the integrand is not merely $2\pi T\bZ^d$
periodic (in view of its terms) but is $2\pi \bZ^d$ periodic by
virtue of being a sum over all  translates. Thus the integral is
the $(k_1-k_2)$th Fourier coefficient of the sum function. In
particular, with $f=g$ we deduce that the row vector function
\[
[m^0_f(\xi) \quad m_f^1(\xi)\quad  \cdots \quad m_f^{t-1}(\xi)]
\]
is a unit vector almost everywhere.
 Now the
lemma follows exactly as in the monoscaled case.
\end{proof}

\subsection{Separability and Higher Rank.}
We define a general \textit{higher rank multiresolution analysis}
for a commuting $r$-tuple $(A_1,\dotsc,A_r)$ in $\gld$ to be a
pair $\hrmra$, where $\phi \in \Ltwod$ and $\fV=\{V_k:k\in\bZ^r\}$
is a family of closed subspaces satisfying the $r$-fold version of
the conditions (i),$\dotsc$,(vi) of Definition \ref{def:bmra}.
The $r$-tuple is assumed to take the form
$$
A_1 = A_1'\oplus I_{d_2}\oplus\dots\oplus
I_{d_r},\dotsc,A_r=I_{d_1}\oplus\dots\oplus I_{d_{r-1}}\oplus
A_r',
$$
where $d_1+d_2+\dots+d_r=d$ and $A_i'\in GL_{d_i}(\bR)$ are
dilation matrices (with integer entries).

The simplest way to create such a multiresolution analysis is as
the tensor product
\[
\left( \phi_1\otimes\dotsm\otimes\phi_r,\fV_1\otimes\dotsm\otimes
\fV_r\right)
\]
of MRAs $(\phi_i,\fV_i)$, with
$\fV_k=\left\{V_i^{(k)}:i\in\bZ\right\}$, where
\[
\fV_1\otimes\dotsm\otimes\fV_r=\left\{V_k=V_{k_1}^{(1)}\otimes\dotsm\otimes
V_{k_r}^{(r)}:k=(k_1,\dotsc,k_r)\in\bZ^r\right\}.
\]
Here the scaling function $\phi=\phi_1\otimes\dotsm\otimes\phi_r$
is separable in the sense that there is an evident partition of
the coordinates of $L^2(\bR^d)$ with
$\phi(x)=\phi_1(x)\phi_2(x)\dotsm\phi_r(x)$, where $\phi_i(x)$ is
a function of the coordinates in the $i$th partiton set.

We say that a higher rank MRA $\hrmra$ of rank $r$ is
\textit{purely separable} if it is unitarily equivalent to such an
$r$-fold tensor product.  For $r\geq 3$ there can be partial forms
of separability, which we do not discuss here, while for $r=2$ we
shall simply say that a BMRA $\hrmra$ is \textit{separable} if it
is unitarily equivalent to a 2-fold tensor product decomposition.

If $\bmra$ is a BMRA then selecting the diagonally labelled
subspaces $V_{i,i}$ gives rise to the MRA $(\phi,\{V_{i,i}\})$. We
may thus observe from known facts for MRAs that there are
redundancies in the conditions of Definition \ref{def:bmra}. Thus
condition (iii) follows from (i), (ii), (iv), (vi). See Lemma
\ref{intersectunionlem}.
 In general however, we have not found that the presence of MRAs in BMRAs
provides any shortcuts to the construction of higher rank
wavelets. One can develop formulae for filter functions of
functions in $V_k$ and pursue the theory for rank greater than two
but we do not do so here.

\section{Construction of Wavelets from BMRAs}

We now show how to construct wavelets and wavelet sets from a
given BMRA, and we elucidate the interrelationship between filter
functions.

Let $\hrmra$ be a BMRA in $\Ltwod$ for the dilation pair $(A,B)$.
We refine some notation for filter functions as follows.  If $f$
lies in $V_{1,0}$ then we write $m_f^A(\xi)$ for the
$2\pi\bZ^d$-periodic filter function $m_f(\xi)$ arising from
$T=A$, and if $f$ lies in $V_{0,1}$ we write $m_f^B(\xi)$ for the
filter function for $T=B$.  Thus if $f$ lies in $V_{1,0}\cap
V_{0,1}$ then we have, from \eqref{e:filter},
\begin{align*}
\hat{f}(\xi)&=m_f^A(A^{-1}\xi)\hat{\phi}(A^{-1}\xi),
\intertext{and}
\hat{f}(\xi)&=m_f^B(B^{-1}\xi)\hat{\phi}(B^{-1}\xi).
\end{align*}
In particular these remarks apply to $\phi$ itself and so
\begin{align}
\hat{\phi}(A\xi)&=m_\phi^A(\xi)\hat{\phi}(\xi) \label{scalinga};\\
\hat{\phi}(B\xi)&=m_\phi^B(\xi)\hat{\phi}(\xi) \label{scalingb}.
\end{align}
Put $B\xi$ for $\xi$ in \eqref{scalinga} and use \eqref{scalingb}
to obtain
\begin{align}
\hat{\phi}(AB\xi)&=m_\phi^A(B\xi)m_\phi^B(\xi)\hat{\phi}(\xi).
\intertext{Reciprocally}
\hat{\phi}(BA\xi)&=m_\phi^B(A\xi)m_\phi^A(\xi)\hat{\phi}(\xi).
\end{align}
Thus, if $\hat{\phi}(\xi)$ is nonvanishing almost everywhere then
we obtain the fundamental intertwining relation
\begin{equation}
m_\phi^A(B\xi)m_\phi^B(\xi)=m_\phi^A(\xi)m_\phi^B(A\xi).
\label{intertwine}
\end{equation}

Since $\fV$ is a commuting lattice the subspaces $V_{1,0}\ominus
V_{0,0}$ and $V_{0,1}\ominus V_{0,0}$ are orthogonal and so we may
define
\begin{equation*}
W_{0,0}=V_{1,1}\ominus\left((V_{1,0}\ominus
V_{0,0})\oplus(V_{0,1}\ominus V_{0,0})\oplus V_{0,0}\right).
\end{equation*}
Moreover, let $W_{i,j}=D_A^i D_B^j W_{0,0}$ for $(i,j)\in \bZ^2$,
so that
\begin{equation*}
W_{i,j}=V_{i+1,j+1}\ominus\left((V_{i+1,j}\ominus
V_{i,j})\oplus(V_{i,j+1}\ominus V_{i,j})\oplus V_{i,j}\right).
\end{equation*}
Thus  these spaces are orthogonal.  Moreover since the
intersection of the spaces $V_{i,j}$ is the zero space it follows
that
\begin{equation*}
V_{i+1,j+1}=\mathop{\sum \oplus}_{(m,n)\leq(i,j)} W_{m,n},
\end{equation*}
and since the union of the $V_{i,j}$ is dense, $\Ltwod$ is the
Hilbert space direct sum of all the $W_{m,n}$.

 As in the
monoscaled theory, if an explicit orthonormal basis \\
$\{\psi_1,\dotsc,\psi_s\}$ can be constructed for the subspace
$W_{0,0}$ then this is a wavelet set in the sense of the following
definition.
\begin{defn}\label{X:waveletset}
Let $(A,B)$ be a commuting dilation pair in $GL_d(\bR)$.  Then
$\{\psi_1,\dotsc,\psi_s\}$ is a wavelet set for $(A,B)$ if
\begin{equation*}
\left\{|\det(A^m B^n)|^{\frac{1}{2}}\psi_i(A^m B^nx+k):(m,n)\in
\bZ^2,k\in\bZ^d,1\leq i\leq s\right\}
\end{equation*}
is an orthonormal basis in $L^2(\bR^d)$.
\end{defn}

We now show how one can construct a wavelet set in $W_{0,0}$ by
means of a nested Gram-Schmidt orthogonalisation process and
repeated applications of both directions of the equivalences given
in Lemma \ref{l:filtermatrix}.

It is convenient to introduce the following terminology which
anticipates the construction process. As before we consider a BMRA
$\hrmra$ for the dilation pair $(A,B)$, where $\fV$ satisfies the
lattice condition.  Also $p=\det A$, $q=\det B$.
\begin{defn}
A \emph{wavelet family} for the BMRA $\hrmra$ is a set
\begin{equation*}
\F
=\left\{\psi_1^A,\dotsc,\psi_{p-1}^A,\psi_1^B,\dotsc,
\psi_{q-1}^B,\psi_1,\dotsc,\psi_s\right\}
\end{equation*}
where $s=(p-1)(q-1)$, and $\{\psi_i^A\}$ (respectively
$\{\psi_j^B\}$, respectively $\{\psi_k\}$) is an orthonormal set
whose $\bZ^2$ translates form an orthonormal basis for
$V_{1,0}\ominus V_{0,0}$ (respectively $V_{0,1}\ominus V_{0,0}$,
respectively $W_{0,0}$).
\end{defn}
\begin{defn}\label{filterbank}
A \emph{filter bank} for the BMRA $\hrmra$ is a set of functions
$\tilde{\F}$ in $\Ltwod$,
\begin{equation*}
\tilde{\F}
 = \left\{\phi,\psi_1^A,\dotsc,\psi_{p-1}^A,\psi_1^B,\dotsc,
\psi_{q-1}^B, \psi_1,\dotsc,\psi_s\right\},
\end{equation*}
where the functions  $\psi_i^A\in V_{1,0}\ominus V_{0,0}$,
$\psi_j^B\in V_{0,1}\ominus V_{0,0}$, and $\psi_k \in W_{0,0}$ are
such that

 (i) the associated $(pq-1)\times(pq-1)$  filter
matrix $U_{\tilde{\F}}(\xi)$ for $T=AB$ is unitary almost
everywhere,

 (ii) the $p\times p$ (respectively $q\times q$) filter
matrix $U^A(\xi)$ (respectively $U^B(\xi)$) for the set
$\{\phi,\psi_1^A,\dotsc,\psi_{p-1}^A\}$ and $T=A$ (respectively\\
$\{\phi,\psi_1^B,\dotsc,\psi_{q-1}^B\}$ and $T=B$) is unitary
almost everywhere.
\end{defn}

If $\tilde{\F}$ is a filter bank as above then by (i) and Lemma
\ref{l:filtermatrix} the functions of $\tilde{\F}$ and their
$\bZ^d$ translates provide an orthonormal basis for $V_{1,1}$.
Also, by (ii) and Lemma \ref{l:filtermatrix} applied twice, the
set $ \{\phi,\psi_1^A,\dotsc,\psi_{p-1}^A\}$ (resp.
$\{\psi_1^B,\dotsc, \psi_{q-1}^B\}$) has translates forming an
orthonormal basis for $V_{1,0}$ (resp. $V_{0,1}$). It follows that
the set $\{\psi_1,\dotsc ,\psi_s\}$ has translates which form an
orthonormal basis of $W_{0,0}$ and so this set is a wavelet set.

 Although Theorem \ref{W:waveletthm} is stated as
 an existence theorem, the proof provides a recipe for construction
which we shall carry out in the next section.

\begin{thm}\label{W:waveletthm}
Let $\hrmra$ be a rank 2 MRA with respect to the commuting
dilation pair $(A,B)$.
Then there exists a wavelet set for $(A,B)$.
\end{thm}

\begin{proof}
From the preceding discussion it suffices to construct a filter
bank $\tilde{\F} = \{\phi\} \cup {\F}$.

From $\phi$ and the dilation $T = A$ construct the row vector
valued function of $\xi$ given by the normalised row of translated
functions for $m_\phi^A(\xi)$. This has the form
\[
[m_\phi^A(\xi) ~~~ m_\phi^A(\xi+2\pi A^{-1}d_1) \cdots
m_\phi^A(\xi+2\pi A^{-1}d_{p-1})]
\]
where $d_1,\dots , d_{p-1}$ in $\bZ^p \times \{0\}$, together with
$d_0 = 0$, give a set of digits for $A$. Precisely as in the
monoscaled theory we may apply the Gram-Schmidt process to any
full rank $p \times p$ completion of this row (by rows which are
similarly translates of their first entry) to obtain a $p\times p$
unitary matrix-valued function $U(\xi)$. We thus obtain $U(\xi) =
U_\G(\xi)$ for a family $\G = \{\phi,
\psi_1^A,\dotsc,\psi_{p-1}^A\}$ where each $\psi^A_i$ lies in
$V_{1,0} \ominus V_{0,0}$. By Lemma \ref{l:filtermatrix}, these
functions
 are orthonormal, with translates forming an orthonormal
 basis of $V_{1,0} \ominus V_{0,0}$.

In a similar way, using  $\phi,$ the dilation $ T=B$ and a set of
digits in $\{0\}\times \bZ^q$ for $B$, construct a unitary matrix
$U_\K(\xi)$ and orthonormal set  $\K =
\{\phi,\psi_1^B,\dotsc,\psi_{q-1}^B\}$ for which the functions
$\psi^B_j$ have translates forming an orthonormal basis for
$V_{0,1}\ominus V_{0,0}$.

Consider now the union,
\[
\F_{AB} =
\{\phi,\psi_1^A,\dotsc,\psi_{p-1}^A,\psi_1^B,\dotsc,\psi_{q-1}^B\}.
\]
Since $\fV$ is a lattice the $\psi^A_i$ and the $\psi^B_j$ are
orthogonal. Moreover elements of $\F_{AB}$ have
 orthonormal translates in
$V_{1,1}\ominus V_{0,0}$. It thus follows from Lemma
\ref{l:filtermatrix} again that for $T=AB$ the  $(p+q-1)\times pq$
 filter matrix for $\F_{AB}$ and $T$, denoted $U_{\F_{AB}}(\xi)$, is a
partial isometry almost everywhere.

As before we may complete $U_{\F_{AB}}(\xi)$  to a unitary $pq
\times pq$ matrix which is the filter matrix of a family $\F_{AB}
\cup \{\psi_1, \dots ,\psi_s\}.$ This is the desired filter bank
and, by Lemma \ref{l:filtermatrix}, yet again, $\{\psi_1, \dots
,\psi_s\}$ is the desired wavelet set.
\end{proof}

\section{Latin Square Wavelets.}

\subsection{Coset Filter Matrices.}
First we introduce a companion matrix $U'_{\F}(\xi)$ for
the filter matrix $U_\F(\xi)$ determined by functions
$\{f_0,\dotsc,f_{t-1}\}$ in $V_{i,j}$ and a fixed dilation unitary
$T$, as given in Definition \ref{d:cosetmatrix}.  This companion
$t\times t$ matrix uses cosets rather than translates and is
unitary if and only if $U_{\F}(\xi)$ is unitary. Furthermore,  and
it takes a particularly simple form in the case of the Latin
Square wavelets. We were unable to find a reference for this
equivalence and so give the detail here.

 Let $T$ be a (possibly weak) dilation matrix, let $t=\det T$ and let
$E_0,\dotsc,E_{t-1}$ be the cosets of $T\bZ^d$. Let
$d_0,\dotsc,d_{t-1}$ be representative digits, with $d_0=0$ and
$E_0=T\bZ^2$. We are interested in the case $T=A^iB^j$ with spaces
$V_{0,0}$ and $V_{i,j}$. For a function $f$ in $V_{i,j}$ we have
the coefficients $c_f(k)$ for $f$ and $T$ as before, determining
the filter function
\begin{equation*}
m_f(\xi)=\frac{1}{t}\sum_{k\in\bZ^d}c_f(k)e^{-i\langle\xi,k\rangle}.
\end{equation*}
Recall that the matrix $U_{\F}(\xi)$ is determined by its first
column which consists of the filter functions $m_{f_l}(\xi), 0 \le
l \le t-1$. The rows are formed by the translates
$m_{f_l}^i(\xi)$, $1\leq i \leq t-1$. Consider the coset sum
\begin{align*}
m_{f_{l,p}}(\xi)&=\frac{1}{t}\sum_{j\in E_p} c_{f_l}(j)e^{-i\langle\xi,j\rangle}\\
&=\left(\frac{1}{\sqrt{t}}\sum_{k\in\bZ^2}c_{f_l}(d_p+Tk)
e^{-i\langle\xi,Tk\rangle}\right)\frac{1}{\sqrt{t}}e^{-i\langle\xi,d_p
\rangle}\\
&=\mu_{f_{l,p}}(\xi)\frac{1}{\sqrt{t}}e^{-i\langle\xi,d_p\rangle},
\end{align*}
where $\mu_{f_{l,p}}(\xi)$ is defined as the bracketed sum.  Then
\begin{align*}
m_{f_{l,p}}^i(\xi)&=m_{f_{l,p}}(\xi+2\pi T^{-1}d_i)\\
&=\mu_{f_{l,p}}(\xi+2\pi T^{-1}d_i)\frac{1}{\sqrt{t}}
e^{-i\langle \xi+2\pi T^{-1}d_i,d_p\rangle}\\
&=\mu_{f_{l,p}}(\xi)D_{p,i}(\xi),
\end{align*}
by the $T^{-1}\bZ^2$ periodicity of $\mu_{f_{l,p}}(\xi)$, where
\begin{equation*}
D_{j,k}(\xi)=e^{i{\langle\xi,d_j\rangle}}\frac{1}{\sqrt{t}}e^{-i\langle2\pi
T^{-1}d_k,d_j}\rangle .
\end{equation*}
Note that $D=(D_{j,k})$ is a unitary valued matrix and
\begin{align*}
m^i_{f_l} =
\sum_{p=0}^{t-1}m^{i}_{f_{{l,p}}}=\sum_{p=0}^{t-1}\mu^{i}_{f_{{l,p}}}D_{p,i}
= (U_\F'D)_{l,i}.
\end{align*}
 Thus
\begin{equation*}
U_{\F}=U'_{\F}D,
\end{equation*}
where $U'_{\F}$, which we call the \emph{coset filter matrix} for
$\F$ and $T$, is defined by
\begin{equation*}
U'_{\F}(\xi)=\left(\mu_{f_{l,p}}(\xi)\right)_{l=1,p=0}^{t,t-1}.
\end{equation*}

\subsection{Latin Square Wavelets}
Let $A=\left[\begin{smallmatrix}3 &0\\0&
1\end{smallmatrix}\right],$ $B=\left[
\begin{smallmatrix}1&0\\0&3 \end{smallmatrix}\right]$ be (weak) dilation matrices
providing a commuting dilation pair $(A,B)$ and unitary dilation
operators $D_A,D_B$ on $\Ltwotwo$.  Let $\phi$ be the
characteristic function of the unit square $[0,1]^2$ in $\bR^2$.
Then $\phi$ and $(A,B)$ generate a BMRA $\hrmra$ in $\Ltwotwo$. In
fact $\V$ is simply the tensor product BMRA of two copies of the
triadic Haar wavelet MRA.  We have
\begin{equation*}
\phi((AB)^{-1}x)=\sum_{i=0}^2\sum_{j=0}^2\phi(x-(i,j)).
\end{equation*}
The distinctiveness of this scaling relation is that
$\phi((AB)^{-1}x)$ is simply a linear combination of translates of
$\phi$ by a set of digits for $T=AB$.  This property, as we shall
see, persists in the wavelets that we construct for $\hrmra$ by
filter matrix completion method of Theorem \ref{W:waveletthm}.  We
have
\begin{align*}
\phi(A^{-1}x)&=\phi(x)+\phi(x_1-1,x_2)+\phi(x_1-2,x_2)
\intertext{and so}
m_{\phi}^{A}(\xi_1,\xi_2)&=\frac{1}{3}(1+e^{-i\xi_1}+e^{-i2\xi_1}).
\end{align*}
Let $(0,0),(1,0),(2,0)$ be the natural set of digits for $A$. Then
the three coset functions for $m_{\phi}^A(\xi)$ are simply the
pure frequency functions
$\tfrac{1}{\sqrt{3}},\tfrac{1}{\sqrt{3}}e^{-i\xi_1},
\tfrac{1}{\sqrt{3}}e^{-i2\xi_1}$.
 To complete the row
\[\frac{1}{\sqrt{3}}\left[\begin{matrix}1&e^{-i\xi_1}&e^{-i2\xi_1}\\
\end{matrix}\right]\]
to a unitary matrix valued function of coset functions we first
complete
\[\frac{1}{\sqrt{3}}\left[\begin{matrix}1&1&1\\ \end{matrix}\right]\]
to a $3\times 3$ unitary scalar matrix.  One such completion gives
the desired completion
\begin{equation*}
U_A'(\xi_1,\xi_2)=\left[
\begin{matrix}\tfrac{1}{\sqrt{3}}&\tfrac{1}{\sqrt{3}}&\tfrac{1}{\sqrt{3}}\\\tfrac{1}{\sqrt{6}}&\tfrac{1}{\sqrt{6}}&\tfrac{-2}{\sqrt{6}}\\
\tfrac{1}{\sqrt{2}}& -\tfrac{1}{\sqrt{2}}&0 \end{matrix}
\right]M_A(\xi)
\end{equation*}
where $M_A(\xi)=\textrm{diag}(1,e^{-i\xi_1}e^{-i2\xi_1})$.  A
similar completion matrix $U'_B(\xi_1,\xi_2)$, with
$M_B(\xi)=\textrm{diag}(1,e^{-i\xi_2}e^{-i2\xi_2})$, is associated
with
$$m_\phi^B(\xi_1,\xi_2) =\frac{1}{3}(1+e^{-i\xi_1}+e^{-i2\xi_1})$$
and the digits $(0,0),(0,1),(0,2)$.

 Rows $2,3$ of the completions
above provide functions $\psi_1^A,\psi_2^A,\psi_1^B,\psi_2^B$ such
that the five functions
$\{\phi,\psi_1^A,\psi_2^A,\psi_1^B,\psi_2^B\}$ are part of a
wavelet family in the sense of Definition \ref{filterbank}.
Moreover, for this orthonormal set the coset functions for digits
for $T=AB$ provide a partial isometry

\begin{tiny}\[ \left[ \begin {array}{ccccccccc}
\tfrac{1}{3}&\tfrac{1}{3}&\tfrac{1}{3}&\tfrac{1}{3}&\tfrac{1}{3}&\tfrac{1}{3}&\tfrac{1}{3}&\tfrac{1}{3}&\tfrac{1}{3}\\\noalign{\medskip}\tfrac{1}{6}\,\sqrt
{2}&\tfrac{1}{6}\,\sqrt {2}&\tfrac{1}{6}\,\sqrt
{2}&\tfrac{1}{6}\,\sqrt {2}&\tfrac{1}{6}\,\sqrt
{2}&\tfrac{1}{6}\,\sqrt {2}&-\tfrac{1}{3}\,\sqrt
{2}&-\tfrac{1}{3}\,\sqrt {2}&-\tfrac{1}{3}\,\sqrt
{2}\\\noalign{\medskip}\tfrac{1}{6}\,\sqrt {6}&\tfrac{1}{6}\,\sqrt
{6}&\tfrac{1}{6}\,\sqrt {6}&-\tfrac{1}{6}\,\sqrt
{6}&-\tfrac{1}{6}\,\sqrt {6}&-\tfrac{1}{6}\,\sqrt
{6}&0&0&0\\\noalign{\medskip}\tfrac{1}{6}\,\sqrt
{2}&\tfrac{1}{6}\,\sqrt {2}&-\tfrac{1}{3}\,\sqrt
{2}&\tfrac{1}{6}\,\sqrt {2}&\tfrac{1}{6}\,\sqrt
{2}&-\tfrac{1}{3}\,\sqrt {2}&\tfrac{1}{6}\,\sqrt
{2}&\tfrac{1}{6}\,\sqrt {2}&-\tfrac{1}{3}\,\sqrt
{2}\\\noalign{\medskip}\tfrac{1}{6}\,\sqrt
{6}&-\tfrac{1}{6}\,\sqrt {6}&0&\tfrac{1}{6}\,\sqrt
{6}&-\tfrac{1}{6}\,\sqrt {6}&0&\tfrac{1}{6}\,\sqrt
{6}&-\tfrac{1}{6}\,\sqrt
{6}&0\\\noalign{\medskip}\end{array}\right]\]
\end{tiny}

\noindent where we have ordered the columns according to the digit
order
\[(0,0),(0,1),(0,2),(1,0),(1,1),(1,2),(2,0),(2,1),(2,2).\]

We may now complete to a $9 \times 9$ unitary matrix,

\begin{tiny}\[ \left[ \begin {array}{ccccccccc}
\tfrac{1}{3}&\tfrac{1}{3}&\tfrac{1}{3}&\tfrac{1}{3}&\tfrac{1}{3}&
\tfrac{1}{3}&\tfrac{1}{3}&\tfrac{1}{3}&\tfrac{1}{3}\\\noalign{\medskip}\tfrac{\sqrt{2}}{6}&\tfrac{\sqrt{2}}{6}&\tfrac{\sqrt{2}}{6}&\tfrac{\sqrt{2}}{6}&\tfrac{\sqrt{2}}{6}&\tfrac{\sqrt{2}}{6}&-\tfrac{\sqrt{2}}{3}&-\tfrac{\sqrt{2}}{3}&-\tfrac{\sqrt{2}}{3}\\\noalign{\medskip}\tfrac{\sqrt{6}}{6}&\tfrac{\sqrt{6}}{6}&\tfrac{\sqrt{6}}{6}&-\tfrac{\sqrt{6}}{6}&-\tfrac{\sqrt{6}}{6}&-\tfrac{\sqrt{6}}{6}&0&0&0\\\noalign{\medskip}\tfrac{\sqrt{2}}{6}&\tfrac{\sqrt{2}}{6}&-\tfrac{\sqrt{2}}{3}&\tfrac{\sqrt{2}}{6}&\tfrac{\sqrt{2}}{6}&-\tfrac{\sqrt{2}}{3}&\tfrac{\sqrt{2}}{6}&\tfrac{\sqrt{2}}{6}&-\tfrac{\sqrt{2}}{3}\\\noalign{\medskip}\tfrac{\sqrt{6}}{6}&-\tfrac{\sqrt{6}}{6}&0&\tfrac{\sqrt{6}}{6}&-\tfrac{\sqrt{6}}{6}&0&\tfrac{\sqrt{6}}{6}&-\tfrac{\sqrt{6}}{6}&0\\\noalign{\medskip}\tfrac{\sqrt{10}}{6}&-\tfrac{\sqrt{10}}{30}&-\tfrac{2\sqrt{10}}{15}&-\tfrac{2\sqrt{10}}{15}&-\tfrac{\sqrt{10}}{30}&\tfrac{\sqrt{10}}{6}&-\tfrac{\sqrt{10}}{30}&\tfrac{\sqrt{10}}{15}&-\tfrac{\sqrt{10}}{30}\\\noalign{\medskip}0&-\tfrac{\sqrt{15}}{30}&\tfrac{\sqrt{15}}{30}&-\tfrac{2\sqrt{15}}{15}&\tfrac{2\sqrt{15}}{15}&0&\tfrac{2\sqrt{15}}{15}&-\tfrac{\sqrt{15}}{10}&-\tfrac{\sqrt{15}}{30}\\\noalign{\medskip}\tfrac{\sqrt{15}}{15}&-\tfrac{\sqrt{15}}{6}&\tfrac{\sqrt{15}}{10}&0&\tfrac{\sqrt{15}}{15}&-\tfrac{\sqrt{15}}{15}&-\tfrac{\sqrt{15}}{15}&\tfrac{\sqrt{15}}{10}&-\tfrac{\sqrt{15}}{30}\\\noalign{\medskip}\tfrac{\sqrt{10}}{10}&0&-\tfrac{\sqrt{10}}{10}&0&\tfrac{\sqrt{10}}{10}&-\tfrac{\sqrt{10}}{10}&-\tfrac{\sqrt{10}}{10}&-\tfrac{\sqrt{10}}{10}&\tfrac{\sqrt{10}}{5}\end
{array} \right] \]\end{tiny}

Just as the dilated scaling  function $\phi(A^{-1}B^{-1}\xi)$ is a
linear combination of digit translates of $\phi$, so too are the
wavelets, $\psi_1,\psi_2,\psi_3,\psi_4$ which are determined by
the last four rows of the completion. We can confirm and
understand the orthogonality of these wavelets by arranging the
coefficients of
$\tfrac{30}{\sqrt{10}}\psi_1,\tfrac{30}{\sqrt{15}}\psi_2,
\tfrac{30}{\sqrt{15}}\psi_3,\tfrac{10}{\sqrt{10}}\psi_4$ as in the
diagram. One readily sees that the construction creates in this
way a quadruple of latin squares which are pairwise orthogonal,
that is, have vanishing inner products. Such constructs are
natural to study in their own right, and indeed may be used to
provide wavelets which, as here, are entirely natural variants of
Haar wavelets.

In summary, the arguments above have led to the following theorem
where  $\chi_{ij}(x,y)$ denotes the characteristic function of the
set $[0,1/3]^2+(i/3,j/3)$.

\begin{thm}\label{t:latinsquare}
Let $\psi_1, \dots , \psi_4$ be the functions on $\bR^2$ given by
\begin{align*}\psi_1 &=
\frac{\sqrt{10}}{30}(5\chi_{00}-\chi_{01}-4\chi_{02}-4\chi_{10}
-\chi_{11}+5\chi_{12}-\chi_{20}+2\chi_{21}-\chi_{22}),
\\
\psi_2 &= \frac{\sqrt{15}}{30}(-\chi_{01}+\chi_{02}-4\chi_{10}
+4\chi_{11}+4\chi_{20}-3\chi_{21}-\chi_{22}),
\\
\psi_3 &= \frac{\sqrt{10}}{30}(2\chi_{00}-5\chi_{01}+3\chi_{02}
+2\chi_{11}-2\chi_{12}-2\chi_{20}+3\chi_{21}-\chi_{22}),
\\
\psi_4 &= \frac{\sqrt{10}}{10}(\chi_{00}-\chi_{02}+\chi_{11}
-\chi_{12}-\chi_{20}-\chi_{21}+2\chi_{22}).
\end{align*}
Then the set $\{\psi_1,\dotsc,\psi_s\}$ is a bidyadic wavelet set.
That is, the set
\begin{equation*}
\left\{3^{(n+m)/2}\psi_i(3^mx+k_1,3^ny+k_2):(m,n), (k_1,k_2) \in
\bZ^2,1\leq i\leq 4\right\}
\end{equation*}
is an orthonormal basis in $L^2(\bR^2)$.
\end{thm}

\begin{figure} [ht]
\input{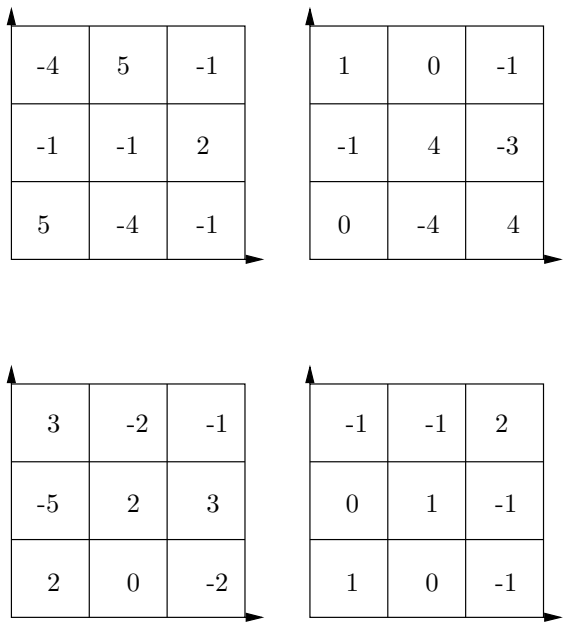}
\caption{Orthogonal Latin squares for the biscaled wavelets
$\psi_1, \dots ,\psi_4$. } \label{latinsquare}
\end{figure}

\section{Dyadic Biscaled Wavelets and Filter Formulae.}

In this section we derive the filter formula, as given in the
introduction, that corresponds to the commuting projection lattice
property of a BMRA.

We start by reproving the well-known fact that an MRA $(\phi,
\{V_i\})$ in $\Ltwo$ for the dyadic dilation matrix $A =[2]$ has
an essentially unique wavelet $\psi_0$. It is determined by the
necessary and sufficient condition that
\[
\hat{\psi}_0(\xi) = m_{\psi_0}(\xi /2)\hat{\phi}(\xi /2)
\]
where the filter $m_{\psi_0}$ for $\psi_0$ for the dilation $T=A$
is given by
\[
m_{\psi_0}= v(\xi)e^{-i\xi}\overline{m_\phi(\xi+\pi)} \] where
$v(\xi)$ is an arbitrary $2\pi$-periodic unimodular function in
$L^\infty (\bR)$.

To see this note that the row matrix function
\[
[m_{\phi}(\xi) \quad m_\phi(\xi +\pi)]
\]
determined by the scaling funciton $\phi$ has a $2 \times 2$
unitary matrix completion of the form
\begin{equation*}
\left[
\begin{matrix} m_\phi(\xi) & m_\phi(\xi+\pi)\\
m_{\psi_0}(\xi) & m_{\psi_0}(\xi + \pi)
\end{matrix} \right]
\end{equation*}
for every function $v(\xi)$ as above. Furthermore,  every unitary
completion in $M_2(L^2(2\pi \bT))$ necessarily has this form. It
follows from Lemma \ref{l:filtermatrix} that each such function
$\psi_0$ has orthonormal translates forming a basis for the
difference space $V_1 \ominus V_0$, and hence that $\psi_0$ is a
wavelet. Conversely, if $\psi'$ is a wavelet, then by Lemma
\ref{l:filtermatrix} $\{\phi, \psi'\}$ necessarily has unitary
filter matrix for $T=A$ and so $\psi'$ necessarily is of the same
form as $\psi_0$.

Consider now a dyadic biscaled wavelet $\psi$ by which we mean a
wavelet for a BMRA $(\phi, \fV)$ in $\Ltwo$ for the dilation
matrices
$$A=\begin{bmatrix}
2&0\\0&1\end{bmatrix},\quad  B=\begin{bmatrix}
1&0\\0&2\end{bmatrix}
$$
obtained from Theorem 3.6. More precisely, since $(\det A -1)(
\det B-1) = 1$ the proof of Theorem 3.6 shows that there exists a
wavelet set which is a singleton $\psi$ whose filter function
$m_\psi^{AB}(\xi)$ for $T = AB$ is associated with a filter bank
$\{\phi, \psi^A, \psi^B, \psi\}$.This filter, $m_\psi^{AB}(\xi)$,
arises from the step-wise unitary completion of the row matrix
function
\[
\left[
\begin{matrix}
 m_\phi^{AB}(\xi) & m_\phi^{AB}(\xi + \pi (1,0))
 & m_\phi^{AB}(\xi + \pi (0,1))& m_\phi^{AB}(\xi +\pi (1,1))
 \end{matrix}\right].
 \]

 By our previous remarks we can
 make explicit how $\psi^A(x)$ and $\psi^B(x)$ may be
 defined. Let us do this and recap the construction process for $\psi$.

Define a function $\psi^A$ by specifying its filter $
m^A_{\psi^A}(\xi) $ for $T=A$ to have the form
\[
m^A_{\psi^A}(\xi) = e^{-i\xi_1}\overline{m_\phi^A(\xi + \pi
(1,0))},
\]
and likewise define $\psi^B$ via its filter function for $T=B$
given by
\[
m^B_{\psi_B}(\xi) = e^{-i\xi_2}\overline{m_\phi^B(\xi + \pi
(0,1))}.
\]
As we have already observed above, the $2 \times 2$-translate
filter matrices for $\{\phi,\psi^A\}$ and $T=A$, and for
$\{\phi,\psi^B\}$ and $T=B$, are unitary almost everywhere. It
follows from Lemma \ref{l:filtermatrix} that $\psi^A$ (resp.
$\psi^B$) has orthonormal translates spanning $V_{1,0}\ominus
V_{0,0}$ (resp. $V_{0,1} \ominus V_{0,0})$. Since these difference
spaces are orthogonal, by the commuting lattice property of a BMRA
subspace grid, we obtain, via Lemma \ref{l:filtermatrix}, a $3
\times 4$ partial isometry valued filter matrix function
$U_\F(\xi)$ for $T=AB$. It suffices to complete this to a $4
\times 4$ unitary-valued filter matrix in order to obtain an
explicit filter function $m_\psi^{AB}$ which then determines the
desired wavelet $\psi$.

\subsection{The Commuting Lattice Filter Relation}
\label{s:commutinglattice} We now examine more directly
the orthogonality of the rows of the $3 \times 4$ filter matrix
for $T=AB$ and $\{\phi, \psi^A, \psi^B\}$.

For convenience we assume that the support of $\hat{\phi}$
contains $[-\pi,\pi]^2$.
 We are now regarding $\psi^A$ as a function in
$V_{1,1}$. This will necessarily have a  filter function
$m^{AB}_{\psi^A}(\xi)$ for $T= AB$ satisfying
\[
\hat{\psi}^A(AB\xi) = m^{AB}_{\psi^A}(\xi)\hat{\phi}(\xi).
\]
We have,
\[
\hat{\psi}^A(A(B\xi)) = m_{\psi^A}^A(B\xi)\hat{\phi}(B\xi)
\\
=m_{\psi^A}^A(B\xi)m_\phi^B(\xi)\hat{\phi}(\xi)
\]
and so, almost everywhere on the support of $\hat{\phi}$, we have
\[
m^{AB}_{\psi^A}(\xi)=m_{\psi^A}^A(B\xi)m_\phi^B(\xi)
\\
=
e^{-i\xi_1}\overline{m_\phi^A(\xi_1+\pi,2\xi_2)}m_\phi^B(\xi_1,\xi_2).
\]
By the support assumption and the $2\pi\bZ^2$ periodicity of the
filters the identitity holds almost everywhere. Similarly,
\[
m^{AB}_{\psi^B}(\xi) =
e^{-i\xi_2}\overline{m_\phi^B(2\xi_1,\xi_2+\pi)}m_\phi^A(\xi_1,\xi_2).
\]
To compactify notation we suppress $\xi_1$ and $\xi_2$ and set,
\[
A_{1,1}^{0,0} = m_\phi^A(\xi_1,\xi_2),\quad  A_{1,1}^{\pi,0} =
m_\phi^A(\xi_1+\pi,\xi_2),
\]
\[
A_{2,1}^{0,0} = m_\phi^A(2\xi_1,\xi_2),\quad  A_{2,1}^{\pi,0} =
m_\phi^A(2(\xi_1+\pi),\xi_2) = A_{2,1}^{0,0},
\]
\[
\overline{A}^{0,\pi}_{2,1} =
\overline{m_\phi^A(2\xi_1,\xi_2+\pi)},
\]
and so on, with $B^{0,0}_{1,1}, B^{\pi,0}_{1,1},\dots $ similarly
associated with $m_\phi^B(\xi_1,\xi_2)$. The $3 \times 4$ filter
matrix for $\{\phi, \psi^A, \psi^B\}$ is the product $DU$ of the
diagonal matrix $D = diag\{1,e^{-i\xi_1}, e^{-i\xi_2}\}$ and the
$3 \times 4$ matrix,
\[
U= \left[
\begin{matrix}
 A_{1,2}^{0,0}B_{1,1}^{0,0} &
 A_{1,2}^{\pi,0}B_{1,1}^{\pi,0}&
 A_{1,2}^{0,0}B_{1,1}^{0,\pi} &
 A_{1,2}^{\pi,0}B_{1,1}^{\pi,pi}\\
  & & & \\
\overline{A}_{1,2}^{\pi,0}B_{1,1}^{0,0}&
-\overline{A}_{1,2}^{0,0}B_{1,1}^{\pi,0}&
\overline{A}_{1,2}^{\pi,0}B_{1,1}^{0,\pi}
&-\overline{A}_{1,2}^{0,0}B_{1,1}^{\pi,\pi}\\
& & & \\
 A_{1,1}^{0,0}\overline{B}_{2,1}^{0,\pi}&
A_{1,1}^{\pi,0}\overline{B}_{2,1}^{0,\pi}&
-A_{1,1}^{0,\pi}\overline{B}_{2,1}^{0,0}&
-A_{1,1}^{\pi,\pi}\overline{B}_{2,1}^{0,0}
\end{matrix}
\right].
\]

The first row of $U$ can be written in the alternate form,
\[\left[
\begin{matrix}
A_{1,1}^{0,0}B_{2,1}^{0,0} &
A_{1,1}^{\pi,0}B_{2,1}^{0,0} &
A_{1,1}^{0,\pi}B_{2,1}^{0,\pi} &
A_{1,1}^{\pi,\pi}B_{2,1}^{0,\pi}
\end{matrix}\right]
\]
in view of the intertwining relations
\[
m_\phi^{AB}(\xi_1, \xi_2) = A_{1,2}^{0,0}B_{1,1,}^{0,0} =
A_{1,1}^{0,0}B_{2,1}^{0,0},\quad \mbox{etc}.
\]

We note that the unitarity almost everywhere of the filter
matrices for $\{\phi, \psi^A\}$ and $\{\phi, \psi^B\}$ implies
that almost everywhere
\[
|A^{0,0}_{1,1}|^2 + |A^{\pi,0}_{1,1}|^2 = 1, \quad
|B^{0,0}_{1,1}|^2 + |B^{0,\pi}_{1,1}|^2 = 1.
\]
Thus, the inner product of rows $1$ and $2$ of $U$ is
\[
A^{0,0}_{1,2}A^{\pi,0}_{1,2}|B^{0,0}_{1,1}|^2 -
A^{\pi,0}_{1,2}A^{0,0}_{1,2}|B^{\pi,0}_{1,1}|^2 +
A^{0,0}_{1,2}A^{\pi,0}_{1,2}|B^{0,\pi}_{1,1}|^2 -
A^{\pi,0}_{1,2}A^{0,0}_{1,2}|B^{\pi,\pi}_{1,1}|^2
\]
\[
=A^{0,0}_{1,2}A^{\pi,0}_{1,2}(|B^{0,0}_{1,1}|^2 +
|B^{0,\pi}_{1,1}|^2) -
A^{\pi,0}_{1,2}A^{0,0}_{1,2}(|B^{\pi,0}_{1,1}|^2 +
|B^{\pi,\pi}_{1,1}|^2) = 0.
\]

Likewise, using the alternative form for row $1$, the rows $1$ and
$3$ of $U$ are orthogonal.

\medskip

To recap, we have a $3 \times 4$ partial isometry translate filter
matrix arising from the $T=AB$ filters for $\phi, \psi^A, \psi^B$,
and
\medskip
\[
m^{AB}_{\psi^A}(\xi) = e^{-i\xi_1}\overline{m_\phi^A(\xi_1 +
\pi,2\xi_2)}m_\phi^B(\xi_1,\xi_2) =
e^{-i\xi_1}\overline{A}_{1,2}^{\pi,0}B_{1,1}^{0,0},
\]
\[
m^{AB}_{\psi^B}(\xi) =
e^{-i\xi_2}m_\phi^A(\xi_1,\xi_2)\overline{m_\phi^B(2\xi_1,\xi_2+\pi)}
=e^{-i\xi_2}{A}_{1,1}^{0,0}\overline{B}_{2,1}^{0,\pi}.
\]
\medskip

As we have remarked, $\psi^A$ and $\psi^B$ lie in $V_{1,0}\ominus
V_{0,0}$ and $V_{0,1}\ominus V_{0,0}$ respectively. This is a
consequence of the commuting lattice property of a BMRA and it is
for this reason that rows 3 and 4 are orthogonal almost
everywhere. Conversely this necessary condition (together with the
previous row orthogonality) is sufficient for the commuting
lattice property. Indeed, the orthogonality of $V_{1,0}\ominus
V_{0,0}$ and $V_{0,1}\ominus V_{0,0}$  follows from this (via
Lemma 2.7 yet again) and the commuting projection lattice property
follows. The row orthogonality is the following formula, which is
written in expanded form in the introduction.
\medskip

\begin{align*}
A_{1,2}^{\pi,0}A_{1,1}^{0,0}\overline{B_{1,1}^{0,0}B_{2,1}^{0,\pi}}
-
A_{1,2}^{0,0}A_{1,1}^{\pi,0}\overline{B_{1,1}^{\pi,0}B_{2,1}^{0,\pi}}
- &\\
A_{1,2}^{\pi,0}A_{1,1}^{0,\pi}\overline{B_{1,1}^{0,\pi}B_{2,1}^{0,0}}
+
A_{1,2}^{0,0}A_{1,1}^{\pi,\pi}\overline{B_{1,1}^{\pi,\pi}B_{2,1}^{0,0}}&=0
\end{align*}
\medskip

\begin{rem} We have shown that for a dyadic BMRA scaling function $\phi$
there are two necessary conditions on the "marginal" filters
$m_\phi^A(\xi)$ and $m_\phi^B(\xi)$, namely the intertwining
condition and the commuting lattice (or orthogonality) condition
above. In the final section we shall construct a  function $\phi$
which satisfies these requirements and which defines a
nonseparable BMRA. Evidently this nonseparable scaling function
construction, \textit{ab initio}, is considerably more complicated
than that of constructing nonseparable wavelets, such as the LAtin
square wavelets, from a given (possibly separable) BMRA.
\end{rem}

\begin{rem}
It would be interesting to pursue a "Riesz theory" of general
not-necessarily-commuting lattices associated with scaling
functions with, perhaps, Riesz basis translates. Going somewhat in
this direction, we remark that the following is true (and the
proof is rather delicate). For a "noncommuting-BMRA", that is a
pair $(\phi, \V)$ satisfying all the axioms for a BMRA except the
commuting lattice axiom (v), the subspace grid is necessarily a
lattice.
\end{rem}

\section{Impossibility of compact support for
nonseparable BMRA scaling functions} In this section we show that
if the scaling function of a BMRA $(\phi , \fV)$ in $\Ltwo$ is
compactly supported then $\phi(x,y)$ is separable.

\begin{lem}\label{trig} Let $a(\xi_1,\xi_2),b(\xi_1,\xi_2)$ be
non-zero trigonometric polynomials with frequencies in $\bZ^2$.
Suppose that $\alpha, \beta \geq 2$ are integers and for all
$\xi_1 , \xi_2$,
\begin{equation}
a(\xi_1,\beta\xi_2)b(\xi_1,\xi_2)=a(\xi_1,\xi_2)b(\alpha
\xi_1,\xi_2).
\end{equation}
Then $a(\xi_1,\xi_2)=a(\xi_1)$ and $b(\xi_1,\xi_2)=b(\xi_2)$ for
some single variable trigonometric polynomials
$a(\xi_1),b(\xi_2)$.
\end{lem}
\begin{proof}
Write
\begin{align}
a(\xi_1,\xi_2)&=\sum_{j=L_1^a}^{M_1^a}\sum_{k=L_2^a}^{M_2^a}
a_{j,k}e^{i(j\xi_1+k\xi_2)},\label{CS1} \intertext{and}
b(\xi_1,\xi_2)&=\sum_{j=L_1^b}^{M_1^b}\sum_{k=L_2^b}^{M_2^b}
b_{j,k}e^{i(j\xi_1+k\xi_2)},\label{CS2},
\end{align}
where
$\left[L_1^a,M_1^a\right]\times\left[L_2^a,M_2^a\right]=:Q_a$ and
$\left[L_1^b,M_1^b\right]\times\left[L_2^b,M_2^b\right]=:Q_b$ are
the minimal rectangles containing the support of the coefficients
$a_{j,k},b_{j,k}$ respectively.  Also define $a_{j,k},b_{j,k}$ to
be zero outside their respective rectangles. For given $p,q$ the
$(p,q)$th term of $a(\xi_1,\beta\xi_2)b(\xi_1,\xi_2)$ is
\begin{align}
&\sum_{j=-\infty}^{\infty}\sum_{k=-\infty}^{\infty}
a_{p-j,k}b_{j,q-\beta k},\label{CS3} \intertext{while the
$(p,q)$th term of $a(\xi_1,\xi_2)b(\alpha\xi_1,\xi_2)$ is}
&\sum_{j=-\infty}^{\infty}\sum_{k=-\infty}^{\infty} a_{p-\alpha
j,k}b_{j,q-k}.\label{CS4}
\end{align}

Consider the $(p,\beta M_2^a+M_2^b)$th coefficient of
$a(\xi_1,\beta\xi_2)b(\xi_1,\xi_2)$. Since
\[
a_{p-j,k}b_{j,q-\beta k}= a_{p-j,k}b_{j,\beta(M_2^a-k)+M_2^b}
\]
this term is nonzero only if $k\leq M_2^a$ and $\beta(M_2^a-k)
\leq 0$, that is, only if $k=M_2^a$. Thus the Fourier coefficient
is simply
\begin{equation}
\sum_{j=-\infty}^{\infty}a_{p-j,M_2^a}b_{j,M_2^b}.\label{CS5}
\end{equation}

On the other hand, the $(p,\beta M_2^a+M_2^b)$th term of
$a(\xi_1,\xi_2)b(\alpha\xi_1,\xi_2)$ is
\begin{equation}
\sum_{k=-\infty}^{\infty}\sum_{j=-\infty}^{\infty}a_{p-\alpha
j,k}b_{j,\beta M_2^a+M_2^b-k}.\label{CS5}
\end{equation}
There are nonzero terms in this sum only if $k\leq M_2^a$ and $
\beta M_2^a+M_2^b-k \leq M_2^b$. Thus all terms, and the Fourier
coefficient, are zero if $M_2^a \neq 0$.

Assume, by way of contradiction that this is the case, so that, by
the assumed identity $a(\xi_1,\xi_2)b(\alpha \xi_1,\xi_2)=
a(\xi_1,\beta\xi_2)b(\xi_1,\xi_2)$ we have
\begin{equation}
\sum_{j=\infty}^{\infty}a_{p-j,M_2^a}b_{j,M_2^b}=0.\label{CS6}
\end{equation}

For the case $p=M_1^a+M_1^b$ equation \eqref{CS6} implies
\begin{equation}
a_{M_1^a,M_2^a}b_{M_1^b,M_2^b}=0, \label{CS7}
\end{equation}
hence at least one of $a_{M_1^a,M_2^a},b_{M_1^b,M_2^b}$ is zero.
Define
\begin{align}
z_a&=\max_{j\in [L_1^a,M_1^a]}\left\{j:a_{j,M_2^a}\neq 0\right\},
\intertext{and} z_b&=\max_{j\in
[L_1^b,M_1^b]}\left\{j:b_{j,M_2^b}\neq 0\right\}.
\end{align}
Let $p=z_a+z_b$.
We may now rewrite \eqref{CS6} as
\begin{equation}
\sum_{j=-\infty}^{z_b-1}a_{p-j,M_2^a}b_{j,M_2^b}+a_{z_a,M_1^a}
b_{z_b,M_1^b}+\sum_{j=z_b+1}^{\infty}a_{p-j,M_2^a}b_{j,M_2^b}=0,\label{CS8}
\end{equation}
 For $j>z_b$, $b_{j,M_2^b}=0$ by the
definition of $z_b$.  For $j<z_b$, $p-j>z_a$ so by the definition
of $z_a$, $a_{p-j,M_2^a}=0$. It then follows from \eqref{CS8}
either $a_{z_a,M_2^a}$ or $b_{z_b,M_2^b}$ is zero, which is a
contradiction, and so we must have $M_2^a=0$.

An analogous argument to the one just given,
 beginning with consideration of $q=\beta L_2^a+ L_2^b$, gives
$L_2=0$ and so
\begin{equation}
a(\xi_1,\xi_2)=\sum_{j=L_1^a}^{M_1^a}a_{j,0}e^{ij\xi_1},
\end{equation}
Exchanging roles of the variables it follows that $b(\xi_1,\xi_2)$
is independent of $\xi_1$, as required.
\end{proof}

\begin{thm}
Let $\hrmra$ be a BMRA with respect to dilation pair $(A,B)$ with
scaling function $\phi\in\Ltwotwo$ of compact support. Then $\phi$
is separable.
\end{thm}
\begin{proof}
 Recall that for a BMRA $\hrmra$ with respect to
 dilation pair $(A,B)$, for $(\xi_1,\xi_2)$ we have the
 fundamental intertwining relation
\begin{equation*}
m_\phi^A(\xi_1,\beta\xi_2)m_\phi^B(\xi_1,\xi_2)=
m_\phi^A(\xi_1,\xi_2)m_\phi^B(\alpha\xi_1,\xi_2).
\end{equation*}
If $\phi$ has compact support, then $\hat{\phi}(\xi)$ is
non-vanishing almost everywhere.  Furthermore $\phi(\alpha
\xi_1,\xi_2)$ and $\phi(\xi_1 ,\beta \xi_2)$ are finite linear
combinations of $\bZ^2$-translates of $\phi$ and so the filters
$m_\phi^A, m_\phi^B$ are trigonometric polynomials.
By Lemma \ref{trig} $m_\phi^A(\xi_1,\xi_2)=f(\xi_1)$ and
$m_\phi^B(\xi_1,\xi_2)=g(\xi_2)$, where $f$ and $g$ are
trigonometric polynomials in one variable. It is routine to check
that $\phi_1(\xi_1):=\phi(\xi_1,0)$, with filter $f(\xi_1)$ gives
rise to a rank 1 univariate MRA with respect to dilation by
$\alpha$. Likewise $\phi_2(\xi_2):=\phi(0,\xi_2)$
 with filter $g(\xi_2)$ gives rise to an MRA with respect to dilation by $\beta$.

Using the filter relation \eqref{scalinga} $N+1$ times gives
\begin{equation*}
\hat{\phi}(\xi_1,\xi_2)=\left(\prod_{n=0}^N
f(\alpha^{-n}\xi_1)\right)\hat{\phi}(\alpha^{-n}\xi_1,\xi_2).
\end{equation*}
As $\phi$ is compactly supported, $\hat{\phi}$ is continuous.
Furthermore $f$ is a trigonometric polynomial, hence Lipschitz,
and so the product \\
$\prod_{n=0}^N f(\alpha^{-n}\xi_1)$ converges
almost uniformly, to $F(\xi_1)$, say.  Hence
\begin{align}
\hat{\phi}(\xi_1,\xi_2)&=F(\xi_1)\hat{\phi}(0,\xi_2).
\intertext{Similarly}
\hat{\phi}(\xi_1,\xi_2)&=G(\xi_2)\hat{\phi}(\xi_1,0),
\end{align}
where $G(\xi_2)=\lim_{N\rightarrow\infty}\prod_{n=0}^N
g(\beta^{-n}\xi_2)$.  Hence, almost everywhere, we have
$F(\xi_1)=\hat{\phi}(\xi_1,0)$, $G(\xi_2)=\hat{\phi}(0,\xi_2)$ and
so $\hat{\phi}(\xi_1,\xi_2)=\hat{\phi_1}(\xi_1)\hat{\phi}(\xi_2)$
almost everywhere as required.
\end{proof}

\section{Higher Rank Meyer BMRAs}
In this section we construct a family of bidyadic BMRAs which
include purely nonseparable examples.  The construction is a
higher rank version of the well-known method used by Meyer to
construct wavelets belonging
to the Schwartz class. In particular
the Fourier transform of the scaling function and the resulting
wavelet have compact support. In fact the separable BMRAs obtained
from the tensor product of two rank-1 dyadic Meyer type MRAs are
included here as a special case.
For our construction the scaling function and wavelet have
discontinuous Fourier transforms which, as we have seen in the
last section, are unavoidable. Thus our wavelets do not lie in the
Schwartz class and it is not immediately obvious to what extent
the decay may be improved.

It is natural, by way of motivation and orientation, to recall the
construction of Meyer wavelets, which we now do.

Suppose that $\phi \in \Ltwo$ is a unit vector with orthonormal
translates which satisfies the scaling relation
\[
\phi(x/2) = \sum_{k\in \bZ} a_k\phi(x-k),
\]
with convergence in $\Ltwo$, and is such that $\hat{\phi}(\xi)$ is
continuous at $0$, with $\hat{\phi}(0) \neq 0$. Then from $\phi$
and the scaling unitary for $A=[2]$ one obtains an MRA $(\phi,
\fV)$. This fact is well-known (\cite{Woj}, Theorem 2.13).

A scaling function $\phi$ of this type may be constructed by
specifying its Fourier transform $\theta(\xi) = \hat{\phi}(\xi)$
to have the following three properties:
\medskip

(i) $\sum_{l \in \bZ}|\theta(\xi + 2\pi l)|^2 = \frac{1}{2 \pi},
\mbox{almost everywhere}. $

This condition is equivalent to the orthonormality of translates.
(See Lemma \ref{l:onset})

(ii) $\theta(2\xi) = \psi(\xi)\theta(\xi),$ for some
$2\pi$-periodic function $\psi(\xi)$. This  is equivalent to the
scaling relation above.

(iii) $\theta(\xi)$ is continuous at $0$ with $\theta(0) \neq 0.$
\medskip

To construct such a function $\theta(\xi)$ one may take the
following route of Meyer and construct first a nonnegative
function $\theta$ on $\bR$ which is symmetric on $[-2\pi , 2\pi]$,
with
\[
\theta(\xi)^2 + \theta(\xi-2\pi))^2 = \frac{1}{2\pi},  \quad
\mbox{ on } [0,2\pi ],
\]
\[
\theta(\xi) = \frac{1}{\sqrt{2\pi}}, \quad \mbox{ for } |\xi| <
\frac{2\pi}{3}, \]
\[ \theta(\xi) = 0,  \quad \mbox{
for } |\xi| > \frac{4\pi}{3},
\]
Thus (i) holds, with at most two nonzero summands for each $\xi$.
Let $f(\xi)$ be the $2\pi$-periodic extension of
$\sqrt{2\pi}\theta(2\xi)$ for $\xi \in [-\pi, \pi]$. Then it
follows that the scaling relation (ii) holds. If in addition
$\theta(\xi)$ is continuous at $0$ with $\theta(0) \neq 0$ then
the construction is complete.

It is completely elementary to construct a function $\theta$ on
$\bR$ with the properties above. The main point in the
construction is that, firstly, since $\theta (\xi)
=\frac{1}{\sqrt{2\pi}}$, for $|\xi | \leq \frac{2\pi}{3},$ we have
the scaling relation
\[
f(\xi)\theta(\xi) = \sqrt{2\pi }\theta(2\xi)\frac{1}{\sqrt{2\pi}}
= \theta(2\xi),
\]
which holds in fact for the bigger range $|\xi| \leq \pi $ since
$\theta(2\xi), $ and hence $f(\xi)$, are zero in the range
$\frac{2\pi}{3} \leq |\xi| \leq \pi$. Thus, there is no obstacle
to periodically extending $f(\xi)$ to a function on $\bR$ and
maintaining the scaling relation (ii).

We are going to follow a similar procedure to construct a bidyadic
scaling function in $\Ltwo$ which determines a multiresolution for
$A=\left[\begin{smallmatrix}2 &0\\0& 1\end{smallmatrix}\right],$
$B=\left[ \begin{smallmatrix}1&0\\0&2
\end{smallmatrix}\right]$. However,
while one can readily construct a nonnegative function
$\theta(\xi, \xi_2)$ on $[-2\pi ,2\pi ]^2$ with the properties
\[
\theta(\xi) = \frac{1}{{2\pi}} \mbox{ for } \xi \in
(-\frac{2\pi}{3},\frac{2\pi}{3})^2
\]
\[ \theta(\xi) = 0  \mbox{
for } \xi \notin [-\frac{4\pi}{3},\frac{4\pi}{3}]^2
\]
and
\[ \theta(\xi)^2+\theta(\xi-2\pi (1,0))^2+\theta(\xi-
2\pi(0,1))^2+\theta(\xi -2\pi (1,1))^2=\frac{1}{4\pi^2},
\]
there is no guarantee that one can periodically extend the
functions $\psi^A(\xi), \psi^B(\xi)$ defined on the support of
$\theta$ by
\[
\psi^A(\xi) = \frac{\theta(2\xi_1,\xi_2)}{\theta (\xi_1,
\xi_2)},\quad  \psi^B(\xi) = \frac{\theta(\xi_1,2\xi_2)}{\theta
(\xi_1, \xi_2)}.
\]

Our first main task is to construct $\theta$ with extra structure
so that this will be possible. This step is necessary because as
we have seen in the general theory, if $\phi$ and $A, B$ provide a
BMRA then, by virtue of the subspace inclusions, $\phi$ will have
periodic filter functions $m_\phi^A(\xi)$ and $m_\phi^B(\xi)$
satisfying the intertwining relation. In fact we are arguing here
in the reverse direction. We construct periodic extensions
$\psi^A(\xi), \psi^B(\xi)$. These will be the filters for $\phi$ ,
since
\[
\hat{\phi}(A\xi) = \theta(2\xi_1,\xi_2) = \psi^A(\xi)\theta(\xi) =
\psi^A(\xi)\hat{\phi}(\xi),
\]
\[
\hat{\phi}(B\xi) = \theta(\xi_1,2\xi_2) = \psi^B(\xi)\theta(\xi) =
\psi^B(\xi)\hat{\phi}(\xi).
\]
The intertwining condition follows from the equalities
\[
\psi^A(\xi_1,2\xi_2)\psi^B(\xi_1,\xi_2) = \left(
\frac{\theta(2\xi_1,2\xi_2)}{\theta(\xi_1,2\xi_2)}\right) \left(
\frac{\theta(\xi_1,2\xi_2)}{\theta(\xi_1,\xi_2)}\right)
\]
\[=
\frac{\theta(2\xi)}{\theta(\xi)} = \left(
\frac{\theta(2\xi_1,\xi_2)}{\theta(\xi_1,\xi_2)}\right) \left(
\frac{\theta(2\xi_1,2\xi_2)}{\theta(2\xi_1,\xi_2)}\right) =
\psi^A(\xi_1,\xi_2)\psi^B(2\xi_1,\xi_2).
\]
However, such a condition does not yet guarantee the orthogonality
structure of the commuting lattice property and we must construct
$\theta$ with further structure to ensure this. We do this in
Theorem 7.2; we consider
 a function $\phi\in\Ltwotwo$ and identify sufficient
conditions on its Fourier transform $\hat{\phi}$ that ensure that
$\phi, A, B$ determine a BMRA.

Note first the following lemma. This follows from the rank one
case (which is a basic fact; see \cite{HerWei}, \cite{Woj}), since
every BMRA $\hrmra$ contains the MRA $\{V_{i,i}:i \in \bZ\}$.

\begin{lem}\label{intersectunionlem}
Let $\phi\in\Ltwod$ be such that $\{\phi(x-k):k\in\bZ^d\}$ is an
orthonormal set in $L^2(\bR^d)$ spanning the closed subspace
$V_{0,0}$. Let $V_{i,j}=D_A^i D_B^j V_{0,0}$. Then
$\bigcap_{(i,j)\in\bZ^2}V_{i,j}=\{0\}$. If, moreover,
$\hat{\phi}(0) \neq 0$ and $\hat{\phi}(\xi)$ is continuous at 0,
then $\bigcup_{(i,j)\in\bZ^2}V_{i,j}$ is dense in $\Ltwod$.
\end{lem}

 We introduce
notation for four families of rectangles which lie in the big
square $(-\frac{4\pi}{3},\frac{4\pi}{3})^2$.

For $i,j,k \in \{0,1\}$, let
\[
I_{(i,j,k)}=\left((-1)^i2^{i+k}
\frac{\pi}{3},(-1)^i2^{(1-i)+k}\frac{\pi}{3}\right)\times
\]
\[
\left((-1)^j
2^{j+(1-k)}\frac{\pi}{3},(-1)^j2^{(1-j)+(1-k)}\frac{\pi}{3}\right),
\]and
\[J_{(i,j)}=\left(-\frac{\pi}{3}+i(-1)^j\pi,\frac{\pi}{3}+i(-1)^j\pi\right)\times
\]
\[
\left(-\frac{\pi}{3}+(1-i)(-1)^j\pi,\frac{\pi}{3}+(1-i)(-1)^j\pi\right).
\]
These rectangles lie between the big square and the small square
$(-\frac{2\pi}{3},\frac{2\pi}{3})^2$. In particular  $I_{0,0,0}$
is the rectangle which is the image of the north-east square
$(\frac{2\pi}{3},\frac{4\pi}{3})^2$ under $A^{-1}$ and  the
rectangles $I_{i,j,k}$ have similar determinations. Note also that
$J_{0,0}$ is the closure of the union of the disjoint rectangles
$A^{-k}((\frac{2\pi}{3},\frac{4\pi}{3})^2)$, $k=1,2,\dots $.

\begin{figure} [ht]
\input{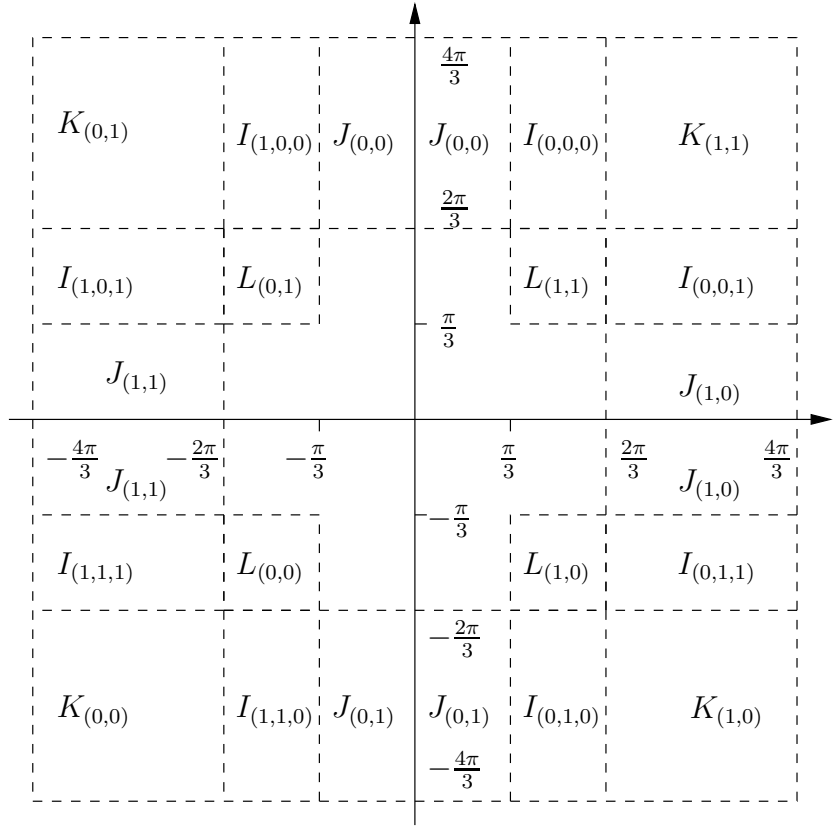}
\caption{Border rectangles in the support of $\hat{\phi}$}
\label{meyer}
\end{figure}

Also for $i,j = 0,1$ let
\[ K_{i,j}=\left( \frac{-4\pi}{3}+i(2\pi),\frac{-2\pi}{3}
+i(2\pi)\right)\times \left(\frac{-4\pi}{3}+j(2\pi),
\frac{-2\pi}{3}+j(2\pi)\right), \]
\[ L_{i,j}=\left( \frac{-2\pi}{3}+i(\pi),\frac{-\pi}{3}+
i(\pi)\right)\times \left(\frac{-2\pi}{3}+j(\pi),
\frac{-\pi}{3}+j(\pi) \right). \] Thus $L_{0,0}$ lies in the south
west corner of the small square.

In the following theorem $\hat{\phi}$ is assumed to be supported
on the big square. Condition (d) is a condition on the restriction
of $\hat{\phi}$ to the four corner squares (translates of
$(\frac{2\pi}{3},\frac{4\pi}{3})^2)$. Conditions (e) and (f) show
how $\hat{\phi}$ is determined on the border rectangles by the
values of $\hat{\phi}$ on pairs of corner squares. Condition (g)
is an additional condition on the restrictions to corner squares.

Despite the detail in conditions (a)-(g) the construction of
examples of such functions $\hat{\phi}$  is quite elementary.
(Also there is further flexibility to arrange $\phi$ to be
real-valued.) Indeed one may define $\hat{\phi}$ on the four
corner squares to comply with (d) and (g), and  then $\hat{\phi}$
is constructed (and uniquely determined) by the  conditions
(a),(b),(f). For example, in Figure 3 we show the regions of
constancy of a function $\hat{\phi}$ which takes constant values
on the trangular subsets of the corner squares. Condition (d) is
elementary and (g) holds trivially with both products in (g)
identically zero. It is evident, from the triangularity of support
in the corners, that $\hat{\phi}$ and hence ${\phi}$ are not
separable.

\begin{center}
\begin{figure}[h]
\centering
\includegraphics[width=8cm]{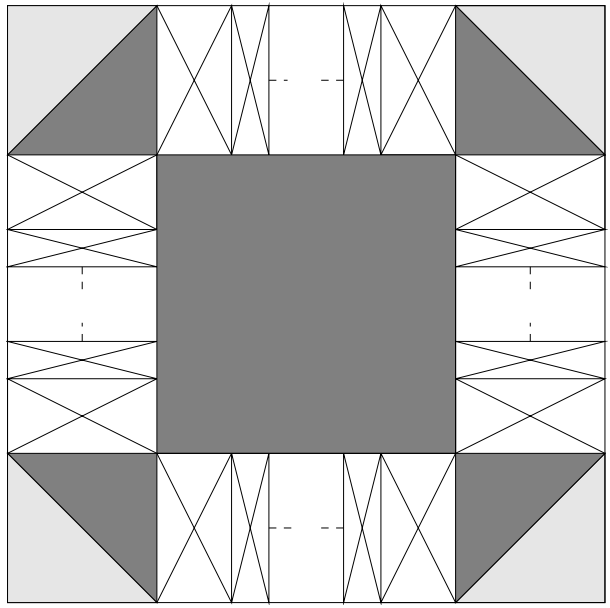}
\caption{Constancy regions of a simple $\hat{\phi}$ example.}
\end{figure}
\end{center}

\begin{thm}\label{t:rank2meyer}
Let $\phi \in \Ltwotwo$ satisfy the following properties :\\
(a) For $(\xi_1,\xi_2)\in\hat{\bR}^2$ we have
$0\leq\hat{\phi}(\xi_1,\xi_2)\leq\tfrac{1}{2\pi}$.\\
(b) For $(\xi_1,\xi_2) \in
(-\tfrac{2\pi}{3},\tfrac{2\pi}{3})^2$ we have
$\hat{\phi}(\xi_1,\xi_2)=\tfrac{1}{2\pi}$.\\
(c) For $(\xi_1,\xi_2) \not\in
(-\tfrac{4\pi}{3},\tfrac{4\pi}{3})^2$we have
$\hat{\phi}(\xi_1,\xi_2)=0$.\\
(d)  For $(\xi_1,\xi_2)\in (\tfrac{2\pi}{3},\tfrac{4\pi}{3})^2$,
$\hat{\phi(\xi)}$ is nonzero and
$$\begin{aligned}& \sum_{i,j\in \{0,1\}}|\hat{\phi}(\xi_1-2\pi
i,\xi_2 -2\pi j)|^2=\frac{1}{4\pi^2}.\end{aligned}
$$
(e) For
$(\xi_1,\xi_2)\in I_{(i,j,k)}$
   and $i,j,k \in \{0,1\}$,
   $$
   \begin{aligned}\hat{\phi}(\xi_1,\xi_2)^2=
   (\frac{1}{2\pi})^2{\frac{\theta_1(\xi_1,\xi_2)^2}
   {\theta_1(\xi_1,\xi_2)^2+\theta_2(\xi_1,\xi_2)^2}}\end{aligned}$$
 where
$$
\begin{aligned}
\theta_1(\xi_1,\xi_2)&=\hat{\phi}(2^{(1-k)}\xi_1,2^k \xi_2),\\
\theta_2(\xi_1,\xi_2)&=\hat{\phi}(2^{(1-k)}\xi_1-2\pi(-1)^i
k,2^k\xi_2-2\pi(-1)^j(1-k))\end{aligned}$$
 (f)  For $(\xi_1,\xi_2)
\in J_{(i,j)}$ and $i,j\in \{0,1\}$,
$$\hat{\phi}(\xi_1,\xi_2)=\hat{\phi}(2^{(1-i)}\xi_1,2^i\xi_2).
$$
(g) For $(\xi_1,\xi_2) \in (\tfrac{2\pi}{3},\tfrac{4\pi}{3})^2$,

$$\hat{\phi}(\xi_1,\xi_2)\hat{\phi}(\xi_1-2\pi,
\xi_2-2\pi)^2-\hat{\phi}(\xi_1-2\pi,\xi_2)
\hat{\phi}(\xi_1,\xi_2-2\pi)^2=0.$$

Let $V_{0,0}$ be the closed subspace spanned by $
\{\phi(\xi_1-k_1,\xi_2-k_2):k_1,k_2\in\bZ\}$, let
$V_{i,j}=D_A^iD_B^jV_{0,0}$ and let $\fV= \{V_{i,j}:i,j\in \bZ\}$.
Then $\hrmra$ is a BMRA with respect to the dilation pair $A,B$.
\end{thm}
\begin{proof}
We show that the conditions of Definition \ref{def:bmra} hold. Let
$V_{i,j}$ be as above. Then  part (ii) of Definition
\ref{def:bmra} is automatically satisfied. By Lemma
\ref{intersectunionlem}, parts (iii) and (iv) of Definition
\ref{def:bmra} will follow from (vi). Let
\[
S(\xi_1,\xi_2)=\sum_{(k_1,k_2)\in \bZ^2}|\hat\phi(\xi_1-2\pi
k_1,\xi_2-2\pi k_2)|^2,
\]
then (iv) will follow from Lemma \ref{l:onset}
 if we  show that $S(\xi_1,\xi_2)=\tfrac{1}{4\pi^2}$ almost
everywhere. It is immediate from (b) and (c) of the theorem that
this is the case for $(\xi_1,\xi_2)\in
(\tfrac{-2\pi}{3},\tfrac{2\pi}{3})^2+2\pi(k_1,k_2)$. Likewise from
(d) it follows for $(\xi_1,\xi_2)\in
(\tfrac{2\pi}{3},\tfrac{4\pi}{3})^2+2\pi(k_1,k_2)$. For
$(\xi_1,\xi_2)\in
I_{(0,0,0)}=(\tfrac{\pi}{3},\tfrac{2\pi}{3})\times
(\tfrac{2\pi}{3},\tfrac{4\pi}{3})$ we have, from (e),
\begin{align*}
S(\xi_1,\xi_2)=&|\hat{\phi}(\xi_1,\xi_2)|^2+|\hat{\phi}(\xi_2,\xi_2-2\pi)|^2\\
=&\frac{1}{4\pi^2}\frac{\hat{\phi}(2\xi_1,\xi_2)^2}{\hat{\phi}(2\xi_1,\xi_2)^2+
\hat{\phi}(2\xi_1,\xi_2-2\pi)^2}\\
&+\frac{1}{4\pi^2}\frac{\hat{\phi}(2\xi_1,\xi_2-2\pi)^2}{\hat{\phi}(\xi_1,\xi_2)^2
+\hat{\phi}(\xi_1,\xi_2-2\pi)^2} \\
 =&\frac{1}{4\pi^2}
\end{align*}
as required. It is straightforward to carry out the preceding
calculation for general $I_{(i,j,k)}$, $i,j,k\in \{0,1\}$ and so
we obtain
$$S(\xi_1,\xi_2)=\tfrac{1}{4\pi^2}\quad  \mbox{ for all }
(\xi_1,\xi_2) \in \{I_{(i,j,k)}+2\bZ^2:i,j,k\in\{0,1\}\}.
$$

Now let $(\xi_1,\xi_2)\in
(0,\tfrac{\pi}{3})\times(\tfrac{2\pi}{3},\tfrac{4\pi}{3})$. Then
there exists $r\in\bN$ such that $\tfrac{\pi}{3}\leq 2^r \xi_1
\leq \tfrac{2\pi}{3}$, and
$$
S(\xi_1,\xi_2)=\hat{\phi}(2^r\xi_1,\xi_2)^2+
\hat{\phi}(2^r\xi_1,\xi_2-2\pi)^2=\tfrac{1}{4\pi^2}.
$$
Again, this calculation may be repeated to show that
$S(\xi_1,\xi_2)=\tfrac{1}{4\pi^2}$ for all $(\xi_1,\xi_2)\in
\{J_{(i,j)}+2\pi \bZ^2:i,j\in \bZ\}$.

 Hence we have shown that
$S(\xi_1,\xi_2)=\frac{1}{4\pi^2}$ almost everywhere, and so
$\hrmra$ satisfies (iii),(iv) and (vi) of Definition
\ref{def:bmra}.

Next we show the BMRA inclusion property (i).  From the definition
of $V_{i,j}$ it suffices to show $V_{-1,0}\subset V_{0,0}$ and
$V_{0,-1}\subset V_{0,0}$. We  show the former. By Proposition
\ref{p:V00} this is equivalent to the equation
\begin{equation*}
\hat{\phi}(2\xi_1,\xi_2)=g(\xi_1,\xi_2)\hat{\phi}(\xi_1,\xi_2)
\end{equation*}
for some $2\pi\bZ^2$ periodic function $g$ with square summable
restriction to $2\pi\bT^2$. We  show such a function exists. The
equality holds trivially regardless of the value taken by $g$ for
$(\xi_1,\xi_2)\notin(\tfrac{-4\pi}{3},\tfrac{4\pi}{3})^2$, hence
if $g$ may be constructed to be $2\pi \bZ^2$ periodic on this
square then its periodic extension will satisfy the equation
everywhere.

 For $(\xi_1,\xi_2)$ such that
$\tfrac{2\pi}{3}<|\xi_1|<\tfrac{4\pi}{3}$,
$\hat{\phi}(2\xi_1,\xi_2)=0$ but $\hat{\phi}(\xi_1,\xi_2)\neq 0$
so that we must take $g(\xi_1,\xi_2)=0$; also
$g(\xi_1,\xi_2)=g(\xi_1-2\pi,\xi_2)=0$ for $(\xi_1,\xi_2)$ such
that $\tfrac{2\pi}{3}<\xi_1<\tfrac{4\pi}{3}$, so periodicity is
preserved. For $(\xi_1,\xi_2)\in
(-\tfrac{2\pi}{3},\tfrac{2\pi}{3})^2$ we have
$\hat{\phi}(\xi_1,\xi_2)=\tfrac{1}{2\pi}$, so we set
$g(\xi_1,\xi_2)=\hat{\phi}(2\xi_1,\xi_2)$; as
$$((-\tfrac{2\pi}{3},\tfrac{2\pi}{3})^2+
2\pi(k_1,k_2))\cap(-\tfrac{4\pi}{3},
\tfrac{4\pi}{3})^2=\emptyset$$ if $(k_1,k_2) \neq (0,0)$,
periodicity is maintained.

 For the remaining $(\xi_1,\xi_2)
\in(\tfrac{-4\pi}{3},\tfrac{4\pi}{3})^2$, both
$\hat{\phi}(\xi_1,\xi_2),$ and $\hat{\phi}(2\xi_1,\xi_2)$ are
nonzero and we may define, as in the rank 1 case described
earlier,
\[
g(\xi_1,\xi_2)=\frac{\hat{\phi}(2\xi_1,\xi_2)}{\hat{\phi}(\xi_1,\xi_2)}.
\]
Then for $(\xi_1,\xi_2)\in I_{(0,0,0)}$ we have
\begin{align*}
g(\xi_1,\xi_2)&=\hat{\phi}(2\xi_1,\xi_2)
\sqrt{\frac{\hat{\phi}(2\xi_1,\xi_2)^2+
\hat{\phi}(2\xi_1,\xi_2-2\pi)^2}{\hat{\phi}(2\xi_1,\xi_2)^2}}\\
&=\sqrt{\hat{\phi}(2\xi_1,\xi_2)^2+\hat{\phi}(2\xi_1,\xi_2-2\pi)^2}\\
&=\hat{\phi}(2\xi_1,\xi_2-2\pi)\sqrt{\frac{(\hat{\phi}(2\xi_1,(\xi_2-2\pi)+2\pi))^2+(\hat{\phi}(2\xi_1,\xi_2-2\pi))^2}{(\hat{\phi}(2\xi_1,\xi_2-2\pi))^2}}\\
&=\frac{\hat{\phi}(2\xi_1,\xi_2-2\pi)}{\hat{\phi}(\xi_1,\xi_2-2\pi)}=g(\xi_1,\xi_2-2\pi)
\end{align*}
so that $g$ is periodic on $I_{(0,0,0)}\cup I_{(0,1,0)}$.  Again
repeating a variation of this  calculation, or by appealing to
symmetry, we obtain periodicity for points in all $I_{(i,j,k)}$.

It remains to consider $(\xi_1,\xi_2)\in J_{(0,0)}\cup J_{(0,1)}$.
Periodicity of $g$ for such $(\xi_1,\xi_2)$ follows from the
recursive definition of $\hat{\phi}(\xi_1,\xi_2)$ on this interval
and the periodicity of $g$ on $I_{i,j,k}$.  Hence $g$ is periodic
and $V_{-1,0}\subset V_{0,0}$. By symmetry it follows that
$V_{0,-1}\subset V_{0,0}$ and the other inclusions in  Definition
\ref{def:bmra} (i) are satisfied.

Finally we must show that the spaces $V_{i,j}$ form a commuting
lattice. It is sufficient to show that $(V_{0,1}\ominus
V_{0,0})\perp (V_{1,0}\ominus V_{0,0})$. As we have discussed in
Section 5.1, we require
\begin{equation}
e^{-i(\xi_1+\xi_2)}f(\xi_1,\xi_2)=0,\label{gcondition}\end{equation}for
almost every $(\xi_1,\xi_2)\in \hat{\mathbb{R}}^2$, where
\begin{align*}f(\xi_1,\xi_2)&=A_{1,2}^{\pi,0}B_{1,1}^{0,0}B_{2,1}^{0,\pi}A_{1,1}^{0,0}-A_{1,2}^{0,0}B_{1,1}^{\pi,0}B_{2,1}^{0,\pi}A_{1,1}^{\pi,0}\\
& \quad
+A_{1,2}^{\pi,0}B_{1,1}^{0,\pi}B_{2,1}^{0,0}A_{1,1}^{0,\pi}+
A_{1,2}^{0,0}B_{1,1}^{\pi,\pi}B_{2,1}^{0,0}A_{1,1}^{\pi,\pi}.\end{align*}
The factors here arise from the filters $m_\phi^A, m_\phi^B$ of
$\phi$. Note that we have determined these filters as ratios, for
example, $m_\phi^A$ is given explicitly by the function $g(\xi_1,
\xi_2)$ above.

Our aim is to show that the unwieldy expression above is
equivalent to condition (g) for a function $\phi$ satisfying
conditions (a)-(f). To that end we first observe that if on some
rectangle at least one term in each of the products present in $f$
vanishes, then $f$ vanishes identically in that rectangle.

 For
$(\xi_1,\xi_2) \in \{K_{i,j}\}_{i,j\in\{0,1\}}$,
$A_{1,2}^{\pi,0}=A_{1,2}^{0,0}=B_{2,1}^{0,\pi}=B_{2,1}^{0,0}=0$,
hence at least one term in each of the products in
$f(\xi_1,\xi_2)$ is zero and so $f(\xi_1,\xi_2)=0$ for  almost
every $(\xi_1,\xi_2) \in \{K_{i,j}\}_{i,j=0,1}$.  For
$(\xi_1,\xi_2)$ such that
$\tfrac{2\pi}{3}<|\xi_2|<\tfrac{4\pi}{3}$,
$A_{1,2}^{\pi,0}=A_{1,2}^{0,0}=0$, and so again for almost all
such $(\xi_1,\xi_2)$ we have $f(\xi_1,\xi_2)=0$.  Likewise, for
$(\xi_1,\xi_2)$ such that
$\tfrac{2\pi}{3}<|\xi_1|<\tfrac{4\pi}{3}$,
$B_{2,1}^{0,\pi}=B_{2,1}^{0,0}=0$ and so $f(\xi_1,\xi_2)=0$ for
such $(\xi_1,\xi_2)$.  Hence we may restrict our attention to
$(\xi_1,\xi_2)\in (-\tfrac{2\pi}{3},\tfrac{2\pi}{3})^2$.
For $(\xi_1,\xi_2)$ in the central
square $(-\tfrac{\pi}{3},\tfrac{\pi}{3})^2$,
$A_{1,1}^{\pi,0}=A_{1,1}^{\pi,\pi}=B_{1,1}^{0,\pi}=
B_{1,1}^{\pi,\pi}=A_{1,2}^{\pi,0}=B_{2,1}^{0,\pi}=0$,
 and hence $f(\xi_1,\xi_2)=0$.  For $(\xi_1,\xi_2)$ such
 that $\tfrac{\pi}{3}<|\xi_2|<\tfrac{2\pi}{3}$ and
 $|\xi_1|<\tfrac{\pi}{3}$, $A_{1,2}^{\pi,0}=A_{1,1}^{\pi,0}=A_{1,1}^{\pi,\pi}
 =0$.  Likewise for $(\xi_1,\xi_2)$ with
 $\tfrac{\pi}{3}<|\xi_1|<\tfrac{2\pi}{3}$ and
 $|\xi_2|<\tfrac{\pi}{3}$, $B_{1,1}^{0,\pi}=B_{1,1}^{\pi,\pi}
 =B_{2,1}^{0,\pi}=0$.  Hence $f$ vanishes for  almost every
  $(\xi_1,\xi_2)$ outside  $L_{i,j\in\{0,1\}}$.

Before proceeding, we observe that, for  almost every
$(\xi_1,\xi_2)\in \mathbb{R}^2$
\begin{align*}
f(\xi_1+\pi,\xi_2)=&A_{1,2}^{2\pi,0}B_{1,1}^{\pi,0}B_{2,1}^{\pi,\pi}A_{1,1}^{\pi,0}-A_{1,2}^{\pi,0}B_{1,1}^{2\pi,0}B_{2,1}^{\pi,\pi}A_{1,1}^{2\pi,0}\\
&+A_{1,2}^{2\pi,0}B_{1,1}^{\pi,\pi}B_{2,1}^{\pi,0}A_{1,1}^{\pi,\pi}+A_{1,2}^{\pi,0}B_{1,1}^{2\pi,\pi}B_{2,1}^{\pi,0}A_{1,1}^{2\pi,\pi};\\
=&A_{1,2}^{0,0}B_{1,1}^{\pi,0}B_{2,1}^{0,\pi}A_{1,1}^{\pi,0}-A_{1,2}^{\pi,0}B_{1,1}^{0,0}B_{2,1}^{0,\pi}A_{1,1}^{0,0}\\
&+A_{1,2}^{0,0}B_{1,1}^{\pi,\pi}B_{2,1}^{0,0}A_{1,1}^{\pi,\pi}+A_{1,2}^{\pi,0}B_{1,1}^{0,\pi}B_{2,1}^{0,0}A_{1,1}^{0,\pi};\\
=&-f(\xi_1,\xi_2),
\end{align*}
where we have used the $2\pi\mathbb{Z}^2$ periodicity of the
filters (implicitly in $A_{1,2}^{0,\pi}=A_{1,2}^{0,0}$ and
$B_{2,1}^{\pi,0}=B_{2,1}^{0,0}$). Similar calculations show that
$f(\xi_1,\xi_2+\pi)=-f(\xi-\pi,\xi_2)$ and $f(\xi_1,\xi_2)=
f(\xi_1+\pi,\xi_2+\pi)$, that is
$e^{-(\xi_1+\xi_2)}f(\xi_1,\xi_2)$ is a $\pi\mathbb{Z}^2$-periodic
function, and so it suffices to check that equation
\eqref{gcondition} holds on $L_{0,0}$ for it to hold on all of the
$L_{i,j}$.

 For $(\xi_1,\xi_2)\in L_{0,0}$, we have, by the
definition of $A_{1,1}^{0,0}$;
\begin{align}
A_{1,1}^{0,0}=\frac{\hat{\phi}(2\xi_1,\xi_2)}{\hat{\phi}
(\xi_1,\xi_2)}&=2\pi\hat{\phi}(2\xi_1,\xi_2)\\
&=\frac{\hat{\phi}(2\xi_1,2\xi_2)}{\sqrt{\hat{\phi}(2\xi_1,2\xi_2)^2
+\hat{\phi}(2\xi_1+2\pi,2\xi_2)^2}}.
\end{align}
At this point we introduce the notation
\begin{equation}
\Phi^{a,b}=\hat{\phi}(2\xi_1+a,2\xi_2+b)^2,
\end{equation}
for $a,b = 0, 2\pi$ so that
\begin{equation}
A_{1,1}^{0,0}=\sqrt{\frac{\Phi^{0,0}}{\Phi^{0,0}+\Phi^{2\pi,0}}}.
\end{equation}
The other scaled filters also simplify. For example
\begin{align*}
A_{1,2}^{\pi,0}=m_A(\xi_1+\pi,2\xi_2)=
\frac{\hat{\phi}(2\xi_1+2\pi,2\xi_2)}
{\hat{\phi}(\xi_1+\pi,2\xi_2)}&=\sqrt{\Phi^{2\pi,0}}
\sqrt{\frac{\Phi^{2\pi,0}+\Phi^{2\pi,2\pi}}{\Phi^{2\pi,0}}}\\
&=\sqrt{\Phi^{0,0}+\Phi^{2\pi,2\pi}}.
\end{align*}
In fact on  $L_{0,0}$ all the scaled filters in the formula for
$f(\xi_1,\xi_2)$ are expressible in terms of $\hat{\phi}$
evaluated at the four points $(2\xi_1,2\xi_2)+(i2\pi,j2\pi),
i,j\in \{0,1\}$. With the other filters similarly expressed this
leads to the equality
\begin{align*}
f(\xi_1,\xi_2)&=\Phi^{0,0}\sqrt{\frac{\left(\Phi^{2\pi,0}+\Phi^{2\pi,2\pi}\right)\left(\Phi^{0,2\pi}+\Phi^{2\pi,2\pi}\right)}{\left(\Phi^{0,0}+\Phi^{2\pi,0}\right)\left(\Phi^{0,0}+\Phi^{0,2\pi}\right)}}\\
&-\Phi^{2\pi,0}\sqrt{\frac{\left(\Phi^{0,0}+\Phi^{0,2\pi}\right)\left(\Phi^{0,2\pi}+\Phi^{2\pi,2\pi}\right)}{\left(\Phi^{0,0}+\Phi^{2\pi,0}\right) \left(\Phi^{2\pi,0}+\Phi^{2\pi,2\pi}\right)}}\\
&-\Phi^{0,2\pi}\sqrt{\frac{\left(\Phi^{2\pi,0}+\Phi^{2\pi,2\pi}\right)\left(\Phi^{0,0}+\Phi^{2\pi,0}\right)}{\left(\Phi^{0,2\pi}+\Phi^{2\pi,2\pi}\right)\left(\Phi^{0,0}+\Phi^{0,2\pi}\right)}}\\
&+\Phi^{2\pi,2\pi}\sqrt{\frac{\left(\Phi^{0,0}+\Phi^{0,2\pi}\right)\left(\Phi^{0,0}+\Phi^{2\pi,0}\right)}{\left(\Phi^{0,2\pi}+\Phi^{2\pi,2\pi}\right)\left(\Phi^{2\pi,0}+\Phi^{2\pi,2\pi}\right)}}.
\end{align*}
We simplify the expression by taking out the factor
\begin{equation*}
g(\xi_1,\xi_2):=\left(\sqrt{\left(\Phi^{0,0}+\Phi^{0,2\pi}\right)\left(\Phi^{0,0}+\Phi^{2\pi,0}\right)\left(\Phi^{0,2\pi}+\Phi^{2\pi,2\pi}\right)\left(\Phi^{2\pi,0}+\Phi^{2\pi,2\pi}\right)}\right)^{-1}.
\end{equation*}
By the definition of $\phi$, the function $g$ is nonvanishing
almost every on $L_{0,0}$ and so condition \eqref{gcondition}
holds if and only if
\begin{align*}
0=&\Phi^{0,0}\left(\Phi^{2\pi,0}+\Phi^{2\pi,2\pi}\right)\left(\Phi^{0,2\pi}+\Phi^{2\pi,2\pi}\right)-\Phi^{2\pi,0}\left(\Phi^{0,0}+\Phi^{0,2\pi}\right)\left(\Phi^{0,2\pi}+\Phi^{2\pi,2\pi}\right)\\
&-\Phi^{0,2\pi}\left(\Phi^{2\pi,0}+\Phi^{2\pi,2\pi}\right)\left(\Phi^{0,0}+\Phi^{2\pi,0}\right)+\Phi^{2\pi,2\pi}\left(\Phi^{0,0}+\Phi^{0,2\pi}\right)\left(\Phi^{0,0}+\Phi^{2\pi,0}\right),
\end{align*}
which, after some routine algebra, simplifies to
\begin{equation*}
\left(\Phi^{0,0}\Phi^{2\pi,2\pi}-\Phi^{2\pi,0}\Phi^{0,2\pi}\right)\left(\Phi^{0,0}+\Phi^{2\pi,0}+\Phi^{0,2\pi}+\Phi^{2\pi,2\pi}\right)=0.
\end{equation*}
We observe that
$$\left(\Phi^{0,0}+\Phi^{2\pi,0}+\Phi^{0,2\pi}+
\Phi^{2\pi,2\pi}\right)=\tfrac{1}{4\pi^2}$$ since on rewriting in
terms of $\phi$  it coincides condition (d). Thus
\eqref{gcondition} holds almost everywhere if and only if, for
almost every $(\xi_1,\xi_2)\in K_{0,0}$,
\begin{equation*}
(\hat{\phi}(\xi_1,\xi_2)\hat{\phi}(\xi_1-2\pi,\xi_2-2\pi))^2-
(\hat{\phi}(\xi_1-2\pi,\xi_2)\hat{\phi}(\xi_1,\xi_2-2\pi))^2=0,
\end{equation*}
as required.
\end{proof}

 In summary, we have the following theorem.

\begin{thm}Let $A=\left[\begin{smallmatrix}2 &0\\0& 1\end{smallmatrix}\right],$
$B=\left[ \begin{smallmatrix}1&0\\0&2 \end{smallmatrix}\right]$.
Then, with respect to dilation by $(A,B)$, there exists a
nonseparable real valued wavelet $\psi$, associated with a
nonseparable BMRA $\hrmra$ in $L^2(\bR^2)$, for which $\hat{\psi}$
has compact support.
\end{thm}

\begin{example}
We finish with some further examples.
 Let
$\{\Xi_{i}\}_{i\in\{0,1\}}$ be a partition of
$\{K_{i,j}\}_{i,j\in\{0,1\}}$ such that

(i) $\{\Xi_{i}\}_{i\in\{0,1\}}$ is invariant under the group $\G$
of translation by elements of $2\pi\mathbb{Z}^2$, when
$\{K_{i,j}\}_{i,j\in\{0,1\}}$ is viewed as a subset of
$\mathbb{R}^2/4\pi\mathbb{Z}^2$,

(ii) if $(\xi_1,\xi_2)\in \Xi_0$, then $(-\xi_1,-\xi_2) \in
\Xi_1$.

Observe that if $\hat{\phi}$ is  chosen to satisfy conditions (d)
and (g) of Theorem 6.2 on $\Xi_0$, then by defining
$\hat{\phi}(\xi_1,\xi_2)=\hat{\phi}(-\xi_1,-\xi_2)$ for
$(\xi_1,\xi_2)\in \Xi_1$  the function $\hat{\phi}$ then satisfies
(d) and (g) on the whole of $K_{0,0}$. Consider the simplest case
with  $\hat{\phi}$  constant on each set $K_{i,j}\cap\Xi_0$ for
${i,j\in\{0,1\}}$.  Thus  set $\hat{\phi}(\xi_1,\xi_2)=\alpha$ for
$(\xi_1,\xi_2) \in (K_{0,0}\cap\Xi_0)$, and
$\hat{\phi}(\xi_1,\xi_2)=\delta$ for $(\xi_1,\xi_2) \in
(K_{1,1}\cap\Xi_0)$ for some constants $\alpha,\delta>0$
satisfying $\alpha^2+\delta^2<\tfrac{1}{8\pi^2}$.  Define
$\hat{\phi}(\xi_1,\xi_2)=\beta$ for $(\xi_1,\xi_2) \in
(K_{1,0}\cap\Xi_0)$, where
\begin{equation*}
\beta=\frac{1}{2}\left(\frac{1}{4\pi^2}-\alpha^2-\delta^2\right)+\frac{1}{2}\sqrt{\left(\frac{1}{4\pi^2}-\alpha-\delta\right)^2-4\alpha\delta},
\end{equation*}
and $\hat{\phi}(\xi_1,\xi_2)=\gamma$ for $(\xi_1,\xi_2) \in
(K_{0,1}\cap\Xi_0)$, where
\begin{equation*}
\gamma=\frac{1}{4\pi^2}-\alpha-\beta-\delta.
\end{equation*}
We have now defined $\hat{\phi}$ on $\Xi_0$ and, as suggested
above,  define $\hat{\phi}(\xi_1,\xi_2)=\hat{\phi}(-\xi_1,-\xi_2)$
for $(\xi_1,\xi_2)\in \Xi_1$. Thus  we have defined $\hat{\phi}$
to satisfy (d), (a) and (g) of Theorem 7.2, and so, as before,
$\hat{\phi}$, and ${\phi}$, are determined and yield a BMRA. Note
that we have $\hat{\phi}(\xi_1,\xi_2)=\hat{\phi}(-\xi_1,-\xi_2)$,
and so the scaling function is real valued.

Notice that $\hat{\phi}$ takes at most $4$ values on each set
$I_{(i,j,k)}$ and these values, together with the values $\alpha,
\beta, \gamma, \delta$ and $\frac{1}{2\pi}$ are the only values
taken by $\hat{\phi}$. In particular, $\hat{\phi}$ is a  finite
linear combination of characteristic functions.
\end{example}



\end{document}